\documentclass[11pt,a4paper,final]{article}

\usepackage{amsmath}    
\usepackage{amsfonts}   
\usepackage{amssymb}    
\usepackage{amsthm}     
\usepackage{mathrsfs}   
\usepackage{graphicx}   
\usepackage{color}
\usepackage{bm}
\usepackage{hyperref}
\usepackage[square,sort,comma,numbers]{natbib}
\usepackage[left=3cm,right=3cm,top=3cm,bottom=3cm]{geometry}
\usepackage[titletoc]{appendix}

\newlength{\vslength}
\newcommand{\iid}{{i.i.d.}}

\newtheorem{theorem}{Theorem}[section]

\newtheorem{lemma}{Lemma}[section]

\newtheorem{proposition}{Proposition}[section]

\newtheorem{corollary}{Corollary}[section]

\theoremstyle{remark}

\newtheorem{example}{\bf Example}[section]

\renewenvironment{proof}{\noindent{\it Proof.}}{\qed}

\newcommand{\score}{\dot{\ell}}
\newcommand{\secscore}{\ddot{\ell}}

\newcommand{\scrD}{{\mathscr D}}

\newcommand{\scrM}{{\mathscr M}}

\def\cF{\mathcal F}
\def\cG{\mathcal G}
\def\cH{\mathcal H}

\def\cN{\mathcal N}

\def\cS{\mathcal S}

\def\cX{\mathcal X}

\newcommand{\bD}{{\bf D}}
\newcommand{\bfe}{{\bf e}}

\newcommand{\bH}{{\bf H}}

\newcommand{\bw}{{\bf w}}

\newcommand{\bx}{{\bf x}}
\newcommand{\bX}{{\bf X}}

\newcommand{\bY}{{\bf Y}}

\newcommand{\bbE}{{\mathbb E}}

\newcommand{\bbG}{{\mathbb G}}
\newcommand{\bbN}{{\mathbb N}}
\newcommand{\bbP}{{\mathbb P}}
\newcommand{\bbR}{{\mathbb R}}

\newcommand{\E}{\mathbb{E}}

\newcommand{\bc}{\begin{center}}
\newcommand{\ec}{\end{center}}
\newcommand{\be}{\begin{equation}}
\newcommand{\ee}{\end{equation}}
\newcommand{\ba}{\begin{array}}
\newcommand{\ea}{\end{array}}
\newcommand{\bean}{\setlength\arraycolsep{2pt}\begin{eqnarray*}}
\newcommand{\eean}{\end{eqnarray*}}
\newcommand{\bea}{\setlength\arraycolsep{2pt}\begin{eqnarray}}
\newcommand{\eea}{\end{eqnarray}}
\newcommand{\ben}{\begin{enumerate}}
\newcommand{\een}{\end{enumerate}}
\newcommand{\bed}{\begin{itemize}}
\newcommand{\eed}{\end{itemize}}

\def\half{\hbox{$1\over2$}}

\begin{document}

\thispagestyle{empty}

\title{\vspace*{-9mm}            
	Bayesian Sparse Linear Regression\\ with Unknown Symmetric Error\footnote{This article has been accepted for publication in {\it Information and Inference} Published by Oxford University Press. The accepted version contains significantly improved results, which is available at \url{https://doi.org/10.1093/imaiai/iay022} }}
\author{Minwoo Chae$^{1}$,  Lizhen Lin$^{2}$ and David B. Dunson$^{3}$\\[2mm]
  {\small\it {}$^{1}$\, Department of Mathematics, The University of Texas at Austin}\\
  {\small\it {}$^{2}$\, Department of Statistics and Data Sciences, The University of Texas at Austin}\\
  {\small\it {}$^{3}$\, Department of Statistical Science, Duke University}
  }

\date{\today}
\maketitle

\begin{abstract}
We study full Bayesian procedures for sparse linear regression when errors have a symmetric but otherwise unknown distribution.
The unknown error distribution is endowed with a symmetrized Dirichlet process mixture of Gaussians.
For the prior on regression coefficients, a mixture of point masses at zero and continuous distributions is considered.
We study behavior of the posterior with diverging number of predictors.
Conditions are provided for consistency in the mean Hellinger distance.
The compatibility and restricted eigenvalue conditions yield the minimax convergence rate of the regression coefficients in $\ell_1$- and $\ell_2$-norms, respectively. 
The convergence rate is \emph{adaptive to both the unknown sparsity level and the unknown symmetric error density} under compatibility conditions.
In addition, strong model selection consistency and a semi-parametric Bernstein-von Mises theorem are proven under slightly stronger conditions.

\end{abstract}


\section{Introduction}

Given  data $(x_1, Y_1), \ldots, (x_n, Y_n)$ consisting of response variables $Y_i \in \bbR$ and covariates $x_i \in \bbR^p$, we consider the following linear regression model:
\be\label{eq:lm}
	Y_i = x_i^T\theta + \epsilon_i, \quad i=1, \ldots, n
\ee
where $\theta \in \bbR^p$ is the unknown regression coefficient and $\epsilon_i$'s are random errors following a density $\eta$.
The  data  are assumed to be generated  from  some true pair $(\theta_0, \eta_0)$, where $\theta_0$ is the true regression coefficient vector and $\eta_0$ is the true error density.
We consider the  high-dimensional setting where  $p$, the number of the predictors and the size of the coefficient vector,  may grow with the sample size $n$, and possibly $p\gg n$.
If $p > n$, model \eqref{eq:lm} is not identifiable due to the singularity of its design matrix, therefore $\theta$ is not estimable unless further restrictions or structures are imposed.
A standard assumption for $\theta$ is the sparsity condition which assumes that most  components of $\theta$ are zero.
For the last two decades, model \eqref{eq:lm} has been extensively studied under various sparsity conditions, in particular through penalized regression approaches such as Lasso and its various variants or extensions \cite{Tibshirani94lasso,fuselasso, zou06, elastic-net}.
Recent advances in MCMC and other computational algorithms  have led to a growing development of Bayesian models   incorporating sparse priors \citep{EG2000, ishwaran2005, mitchell88, castillo2012needles, castillo2015bayesian}.  In general, two classes of sparse priors are often used, the first being the spike and slab type  (see e.g., \cite{castillo2012needles, castillo2015bayesian}),  with some recent work extending to continuous versions \citep{ishwaran2005, VREG2016, VR2016, narisetty2014},   and the other  being continuous shrinkage priors; in particular,  local-global shrinkage priors (see \cite{bhattacharya-dunson, Carvalho01062010, poslson-scott}).

In the literature,  both frequentist and Bayesian, the standard Gaussian error model, in which $\epsilon_i$'s are assumed to be \iid\ Gaussian, is typically adopted, providing substantial computational and theoretical benefits.
Using a squared error loss function, various  penalization techniques are developed. 
Theoretical aspects of  such estimates  have  been explored,  showing recovery of $\theta$ in nearly optimal rate or optimal selection of the true non-zero coefficients \citep{bickel2009, DONOHO01091994, fan-li2001,  johnstone2004, castillo2015bayesian}. More recent theoretical advances assure that  relying on certain desparsifying techniques, asymptotically optimal (or at least honest) confidence sets can be constructed \citep{vandegeer2014}.
These results rely on the assumption of Gaussian errors.

Although some theoretical  properties,  such as consistency and rates of convergence, are robust to misspecification of $\eta$, methods that assume Gaussianity  may still face many serious problems when $\eta$ is non-Gaussian. 
First, although a point estimator may be consistent in nearly optimal rate,  its efficiency is not  satisfactory \citep{van1998asymptotic, kleijn2012bernstein, chae2016semiparametric}.
Also, confidence or credible sets do not provide correct uncertainty quantification under model misspecification \citep{kleijn2012bernstein}.
Furthermore,  misspecification can cause  problems in model selection \citep{grunwald2014inconsistency}, resulting in  serious overfitting.
To avoid model misspecification, a natural remedy is to adopt  a semi-parametric model, which treats $\eta$ as an unknown infinite-dimensional parameter.  For semi-parametric models with \emph{fixed $p$}, \cite{bickel1982adaptive} proposed an adaptive estimator, and \cite{minwoo-thesis, chae2016semiparametric} considered a Bayesian semi-parametric  framework, while deriving a misspecified LAN (local asymptotic normality) condition for proving a semi-parametric Bernstein von-Mises (BVM) theorem. However,  little is known about  theoretical properties of  high-dimensional semi-parametric regression model due to technical barriers.


In this paper, we consider a Bayesian semi-parametric approach for the high-dimensional linear regression model \eqref{eq:lm}.
Specifically, we impose a  sparse prior for $\theta$ and Dirichlet process (DP) mixture prior on $\eta$, and study asymptotic behavior of the full posterior distribution,
for which we have developed substantially new tools and theories.
Our work provides a suite of  new asymptotic results including  \emph{posterior consistency, optimal posterior contraction rates, and strong model selection consistency}. A positive theoretical result states that the convergence rate of the marginal posterior of $\theta$ is \emph{adaptive to both the unknown sparsity level and the unknown symmetric error density} under compatibility conditions on the design matrix. Convergence rate of $\eta$ also depends on the unknown sparsity level. More importantly, we also derive the LAN condition for this model with which the \emph{semi-parametric Bernstein von-Mises theorem and strong model selection consistency} are proved.  
The BVM theorem assures asymptotic efficiency  and provides accurate quantification of uncertainties.
To the best of our knowledge, there is no literature, neither frequentist nor Bayesian, considering a semi-parametric efficient estimator for the high-dimensional linear model \eqref{eq:lm}.
Asymptotic results for high-dimensional Bayesian model selection beyond Gaussian error are also quite novel.
It should be noted that in contrast to current frequentist approaches, it is straightforward to modify computational algorithms developed for sparse linear models to allow unknown symmetric errors (see \cite{kundu2014bayes, minwoo-thesis}).
The additional computational burden for each step depends only on $n$, so it is feasible to construct a semi-parametric Bayes estimator for model \eqref{eq:lm}.

It is worthwhile to mention some of the technical aspects.
In the contexts of semi-parametric efficiency, the most challenging problem is to handle semi-parametric biases arising due to the unknown $\eta$.
These biases vanish if score functions are consistent at a certain rate.
Using the structure of Gaussian mixtures, we prove that the ``no-bias" condition holds if $s_0 \log p \leq n^{1/6-\xi}$ for some $\xi > 0$, where $s_0$ is the number of true non-zero coefficients.
For selection consistency in Bayesian high-dimensional models, it is not uncommon to use exponential moment conditions of certain quadratic forms of score functions.
For models with Gaussian error, it is easy to see that such quadratic forms follow chi-square distributions allowing exponential moments.
For non-Gaussian models, a careful application of the Hanson-Wright inequality \cite{hanson1971bound, wright1973bound} provides similar exponential bounds.

The paper is organized as follows. Section \ref{sec-prior-design} introduces the notation,  priors and  design matrices. In  Section \ref{sec-main}  we summarize our main theorems. Section \ref{sec-proofs} includes the proofs of the main theorems and some important lemmas.
Concluding remarks are given in Section \ref{sec-discussion}.
Some well-known results on bracketing and concentration inequalities  frequently used in proofs are provided in  the appendices.

\section{Prior and design matrix}
\label{sec-prior-design}

\subsection{Notation}

In this subsection, we introduce some of the notation used throughout the paper. 
Dependence on the sample size $n$ is often not made explicit.
For a density $\eta$, let $P_\eta$ be the corresponding probability measure.
For $\theta\in\bbR^p$, $x\in\bbR^p$, $y\in\bbR$ and suitably differentiable density $\eta$, let $\ell_\eta(y) = \log \eta(y)$, $\ell_{\theta,\eta}(x,y) = \ell_\eta(y-x^T\theta)$, $\score_\eta(y) = - \partial \ell_\eta(y) / \partial y$, $\secscore_\eta(y) = \partial^2 \ell_\eta(y) / (\partial y)^2$, $\score_{\theta,\eta}(x,y) = \score_\eta(y-x^T\theta) x$ and $\secscore_{\theta,\eta}(x,y) = \secscore_\eta(y-x^T\theta) xx^T$.
The support of $\theta$ is defined as $S_\theta = \{i\leq p: \theta_i \neq 0\}$, and $s_\theta = |S_\theta|$ is the cardinality of $S_\theta$.
Let $S_0 = S_{\theta_0}$ and $s_0 = |S_0|$.
For given $S \subset \{1, \ldots, p\}$, let $\theta_S = (\theta_i)_{i\in S}\in\bbR^{|S|}$ and $\widetilde \theta_S = (\widetilde\theta_i)_{i=1}^p\in\bbR^p$, where $\widetilde \theta_i = \theta_i$ for $i\in S$ and $\widetilde \theta_i = 0$ for $i \in S^c$.
For $1 \leq q < \infty$, define the $\ell_q$-norm as $\|\theta\|_q = (\sum_{i=1}^p |\theta_i|^q)^{1/q}$ and $\|\theta\|_\infty = \max_i |\theta_i|$.

The Hellinger and total variation metrics between two densities $\eta_1$ and $\eta_2$ with respect to $\mu$ are defined as $d_H^2(\eta_1, \eta_2) = \int (\sqrt{\eta_1} - \sqrt{\eta_2})^2d\mu$ and $d_V(\eta_1,\eta_2) =  \int |\eta_1-\eta_2|d\mu$.
Let $K(\eta_1,\eta_2) = \int \eta_1 \log(\eta_1/\eta_2) d\mu$ be the Kullback-Leibler (KL) divergence.
Let $P_{\theta,\eta}^{(n)}$ (the superscript $(n)$ is often excluded) be the probability measure corresponding to model \eqref{eq:lm}, and $P_0^{(n)} = P_{\theta_0, \eta_0}^{(n)}$.
For a given function $f$, let
\bean
	\bbP_n f &=& \frac{1}{n}\sum_{i=1}^n f(x_i, Y_i)
	\\
	\bbG_n f &=& \frac{1}{\sqrt{n}}\sum_{i=1}^n \Big\{f(x_i, Y_i) - P_0 f(x_i, Y_i)\Big\}.
\eean
For a probability measure $P$, $P f$ denotes the expectation of $f$ under $P$.
Expectation under the true distribution is often denoted by $\E$, and $\E_x f = \int f(x,y) \eta_0(y-x^T\theta_0)dy$.
For a class $\cF$ of real valued functions,
$N(\epsilon, \cF, d)$ and $N_{[]}(\epsilon, \cF, d)$ denote the covering and bracketing numbers \cite{van1996weak} of $\cF$ with respect to a (semi-)metric $d$.
For two real numbers $a$ and $b$, $a\vee b$ and $a \wedge b$ denotes the maximum and minimum of $a$ and $b$, respectively.


\subsection{Prior}

Let $\cH$ be the class of continuously differentiable densities $\eta$ with $\eta(x) = \eta(-x)$ and $\eta(x) > 0$ for every $x \in \bbR$, equipped with the Hellinger metric.
We impose a product prior $\Pi = \Pi_\Theta\times\Pi_\cH$ for $(\theta, \eta)$, where $\Pi_\Theta$ and $\Pi_\cH$ are Borel probability measures on $\Theta=\bbR^p$ and $\cH$, respectively.
We use a mixture of point masses at zero and continuous distributions for $\Pi_\Theta$, and a symmetrized DP mixture of normal distributions \cite{minwoo-thesis, chae2016semiparametric} for $\Pi_\cH$.

Specifically, for a prior $\Pi_\Theta$ on $\theta$, we first select a dimension $s$ from a prior $\pi_p$ on the set $\{0, \ldots, p\}$, next a random set $S \subset \{1, \ldots, p\}$ of cardinality $s$, and finally a set of non-zero values $\theta_S$ from a prior density $g_S$ on $\bbR^{|S|}$.
The prior on $(S,\theta)$ can be formally expressed as
$$
	(S,\theta) \mapsto \pi_p(s) \frac{1}{\binom{p}{s}} g_S(\theta_S) \delta_0(\theta_{S^c}),
$$
where the term $\delta_0(\theta_{S^c})$ refers to the coordinates $\theta_{S^c}$ being zero.
Since  sparsity is imposed by $\pi_p$, the density $g_S$ must have tails at least as heavy as the Laplace distribution for desirable large sample properties,  as is well-studied in the Gaussian error case by \cite{castillo2012needles, van2016conditions}.
Data dependent priors \cite{martin2014asymptotically, martin2014empirical, yang2015computational} placing sufficient mass around $\theta_0$ are also possible.
We consider a product of the Laplace density $g(\theta) = \lambda \exp(-\lambda|\theta|)/2$, and have the following assumptions as in \cite{castillo2015bayesian}:
there are constants $A_1, A_2, A_3, A_4 > 0$ with
\be\label{eq:pi_p_condition}
	A_1 p^{-A_3} \pi_p (s-1) \leq \pi_p(s) \leq A_2 p^{-A_4} \pi_p(s-1), ~~~ s=1, \ldots, p,
\ee
and the scale parameter $\lambda$ satisfies
\be\label{eq:lambda_condition}
	\frac{\sqrt{n}}{p} \leq \lambda \leq \sqrt{n \log p}.
\ee
Some useful examples satisfying \eqref{eq:pi_p_condition} are provided in \cite{castillo2015bayesian, castillo2012needles}.

We use a symmetrized DP mixture of normal prior for $\Pi_\cH$, whose properties and inferential methods are  well-known \cite{minwoo-thesis, chae2016semiparametric}.
Assume positive numbers $\sigma_1 < \sigma_2$ and $M$ are given.
Let $\scrM$ be the set of all Borel probability measures on $[-M,M] \times [\sigma_1,\sigma_2]$, and 
$$
	\overline\scrM = \{F\in\scrM: dF(z,\sigma) = dF(-z,\sigma)\}.
$$
Let $\cH_0$ be the set of all $\eta\in\cH$ of the form 
\be \label{eq:eta-mixture}
	\eta(x) = \int \frac{1}{\sqrt{2\pi}\sigma}\exp\bigg\{-\frac{(x-z)^2}{2\sigma^2}\bigg\} dF(z,\sigma)
\ee
for some $F\in\overline\scrM$. 
It can be easily shown that for some constants $L_k$, $k\leq 3$, and functions $m_k(x) = L_k(1 + |x|^k)$, we have
\be\begin{split}\label{eq:libschitz}
	|\ell_\eta(x+y) - \ell_\eta(x)| \leq |y| m_1(x)
	\\
	|\score_\eta(x+y) - \score_\eta(x)| \leq |y| m_2(x)
	\\
	|\secscore_\eta(x+y) - \secscore_\eta(x)| \leq |y| m_3(x)
\end{split}\ee
for every $x\in\bbR$, small enough $|y|$, and $\eta\in\cH_0$.

For the prior $\Pi_\cH$, we first select a random probability measure $F$ from $\scrD(F_0)$ and symmetrize it by $\overline F = (F + F^-)/2$, where $\scrD(F_0)$ denotes the DP with  base measure $F_0$ and $dF^-(z, \sigma) = dF(-z,\sigma)$.
The  resulting prior $\Pi_\cH$ on $\eta$ of the form \eqref{eq:eta-mixture} is supported on $\cH_0$.
We assume that $F_0$ has a continuous and positive density supported on $[-M,M]\times[\sigma_1,\sigma_2]$, so the resulting symmetrized DP $\overline F$ has full weak support on $\overline\scrM$, that is, every non-empty weakly open subset of $\overline\scrM$ has a positive mass.

\subsection{Design matrix}

Denote the design matrix as $\bX = (x_{ij}) \in \bbR^{n\times p}$, and let $\bX_S = (\bx_j)_{j\in S}$, where the boldface $\bx_j=(x_{1j}, \ldots, x_{nj})^T$ is the $j$th column of $\bX$.
Let $\bY = (Y_i)_{i=1}^n$, $\Sigma = (\Sigma_{ij}) = n^{-1} \bX^T \bX$, and $\Sigma_S = (\Sigma_{ij})_{i,j\in S}$.
We consider a fixed design in this paper, so expectations in notations such as $P_0$ and $\bbG_n$ represent the expectation with respect to $\bY$ only.
It is not difficult to generalize to the case of random design by considering the population covariance matrix instead of the sample covariance matrix (see Section 6.12 of \cite{buhlmann2011statistics}).
Since $p$ can be larger than $n$,  certain identifiability conditions are required for the estimability of $\theta$.
Define the \emph{uniform compatibility number} by
\bean
	\phi^2(s) = \inf \bigg\{\frac{s_\theta \theta^T\Sigma\theta}{\|\theta\|_1^2}
	: 0 < s_\theta\leq s \bigg\}
\eean
for $1 \leq s \leq p$.
Also, define the \emph{restricted eigenvalue (or sparse singular value)} by
\bean
	\psi^2(s) = \inf \bigg\{\frac{\theta^T\Sigma\theta}{\|\theta\|_2^2}
	: 0 < s_\theta\leq s \bigg\}.
\eean
By the definition of compatibility number and restricted eigenvalue, we have
$$
	\|\bX\theta\|_2^2 \geq \frac{n}{s_\theta} \phi^2(s_\theta) \|\theta\|_1^2,
	\quad \textrm{and} \quad
	\|\bX\theta\|_2^2 \geq n \psi^2(s_\theta) \|\theta\|_2^2
$$
for every $\theta \in \bbR^p$.
Also, $\psi(s) \leq \phi(s)$ by the Cauchy-Schwartz inequality.
It is sufficient for the recovery of $\theta$ if compatibility numbers evaluated at $K s_0$ for some constant $K > 0$ are bounded away from zero.
Assumptions on the design matrix through $\psi(s)$ are required for recovery with respect to the $\ell_2$-norm, whereas the numbers $\phi(s)$ suffice for $\ell_1$-reconstruction.
See \cite{van2009conditions} and Section 6.13 of \cite{buhlmann2011statistics} for more details and examples.

In the remainder of this paper, we always assume that $\eta_0 \in \cH_0$, and for some positive constants $\alpha$ and $L$, $p \geq n^\alpha$ and $\sup_{i,j} |x_{ij}| \leq L$.
We use the notation $\lesssim$ for smaller than up to a constant multiplication, where the constant is universal (2, $\pi$, $e$, etc.) or depends only on $M, L, \sigma_1, \sigma_2, \alpha$ and $A_j$, for $j\leq 4$.

\section{Main results}
\label{sec-main}

\subsection{Misspecified LAN}

We first consider a certain type of LAN expansion of the log-likelihood, which is an essential property for the proof of the BVM theorem.
This expansion is also very useful for finding a lower bound on the integrated (or marginal) likelihood.
Let 
$$
	R_n(\theta,\eta) = \prod_{i=1}^n \frac{\eta(Y_i - x_i^T \theta)}{\eta_0(Y_i - x_i^T \theta_0)}
$$
and $L_n(\theta,\eta) = \sum_{i=1}^n \ell_{\theta,\eta}(x_i, Y_i)$.
Let $v_\eta = P_{\eta_0} (\score_\eta \score_{\eta_0})$ and
$$
	V_{n,\eta} = \E \bbP_n (\score_{\theta_0,\eta} \score_{\theta_0, \eta_0}^T) = v_\eta \Sigma.
$$
It can be shown \cite{chae2016semiparametric} that $v_\eta = -P_{\eta_0}(\secscore_\eta)$ for every $\eta\in\cH_0$.
Then, the Taylor expansion of $L_n(\theta,\eta) $ around $\theta_0$  roughly implies that
\be\label{eq:mlan-approx}
	L_n(\theta,\eta) - L_n(\theta_0,\eta)
	\approx \sqrt{n}(\theta-\theta_0)^T \bbG_n \score_{\theta_0,\eta}
	- \frac{n}{2} (\theta-\theta_0)^T V_{n,\eta}(\theta-\theta_0)
\ee
around $\theta_0$.
Since the dimension of $\theta$ may be very high, handling the remainder term of the Taylor expansion is technically demanding.
We call the approximation \eqref{eq:mlan-approx} the \emph{misspecified LAN} \cite{chae2016semiparametric} because the left hand side of \eqref{eq:mlan-approx} is the log-likelihood ratio of the misspecified model $\theta\mapsto P_{\theta,\eta}^{(n)}$. We have the following theorem on the remainder term.

\begin{theorem}[Misspecified LAN] \label{thm:mLAN}
Let $(s_n)$ be a sequence of positive integers and $(\epsilon_n)$ be a real sequence such that $\epsilon_n \rightarrow 0$.
Let $\Theta_n$ be a subset of $\{\theta\in\Theta: s_\theta \leq s_n, \|\theta-\theta_0\|_1 \leq \epsilon_n\}$ and define
\bean
	r_n(\theta,\eta) = L_n(\theta,\eta) - L_n(\theta_0,\eta)
	- \sqrt{n}(\theta-\theta_0)^T \bbG_n \score_{\theta_0,\eta}
	+ \frac{n}{2} (\theta-\theta_0)^T V_{n,\eta}(\theta-\theta_0).
\eean
Then,
\bean
	\E \bigg( \sup_{\theta\in\Theta_n} \sup_{\eta\in\cH_0} |r_n(\theta,\eta)| \bigg)
	\lesssim n \epsilon_n^2 \delta_n 
	+ \epsilon_n \sup_{\theta\in\Theta_n} \|\bX(\theta-\theta_0)\|_2^2,
\eean
where $\delta_n = \sqrt{s_n \log p/n}$.
\end{theorem}

If $s_n \log p = o(n)$, Theorem \ref{thm:mLAN} implies that
\be\label{eq:usual-lan}
	\sup_{\substack{{\theta: s_\theta \leq s_n} \\ {\|\theta-\theta_0\|_1 \leq n^{-1/2} }} } \sup_{\eta\in\cH_0} |r_n(\theta,\eta)| = o_{P_0}(1),
\ee
which corresponds to the classical LAN expansion that holds in a $n^{-1/2}$-neighborhood of $\theta_0$.
In the classical setting where $p$ is fixed, the convergence rate of the marginal posterior distribution of $\theta$ is $n^{-1/2}$, and the asymptotic bias $V_{n,\eta}^{-1}(\bbG_n\score_{\theta_0, \eta} - \bbG_n\score_{\theta_0, \eta_0})$ vanishes \cite{chae2016semiparametric}, as $\eta$ gets closer to $\eta_0$ with arbitrary rate in Hellinger distance.
As a result, \eqref{eq:usual-lan} is sufficient for the BVM theorem, which assures asymptotic efficiency of a Bayes estimator.
In the high-dimensional setting,  however, the convergence rate of the full parameter $\theta$ depends on $s_0$.
In particular, it is shown in the next subsection that the convergence rate with respect to the $\ell_1$-norm is $s_0\sqrt{\log p / n}$, so \eqref{eq:usual-lan} is not sufficient to get a BvM type result.
For the BVM theorem to hold, the remainder term $r_n(\theta,\eta)$  should be ignorable in neighborhoods to which the posterior distribution contracts.
Since the convergence rate of $\|\theta-\theta_0\|_1 = O(s_0 \sqrt{\log p / n})$ and $\|\bX(\theta-\theta_0)\|_2^2 = s_0 \log p$ (see Theorem \ref{cor:theta-rate}), it is sufficient that $s_0^5 (\log p)^3 = o(n)$.
Also, for the BVM theorem, it is sufficient that $\bbG_n \score_{\theta_0, \eta}$ and $V_{n,\eta}$ converge to $\bbG_n \score_{\theta_0, \eta_0}$ and $V_{n,\eta_0}$, respectively, at a certain rate.
The details of the results are given  in Section \ref{ssec:bvm}.
Although \eqref{eq:usual-lan} is not helpful for proving the BVM theorem in the high-dimensional problem, it is still very useful for finding a lower bound on the integrated likelihood, which is utilized for  proving  posterior consistency and contraction rates, as shown in the next subsection.

\subsection{Posterior consistency and contraction rate} \label{ssec:consistency}

Let $\Pi(\cdot|\bD_n)$ be the posterior distribution of $(\theta,\eta)$ given $\bD_n$,
\be\label{eq:posterior}
	\Pi(A\times B|\bD_n) = 	\frac{\int_A\int_B R_n(\theta,\eta) d\Pi_\cH(\eta) d\Pi_\Theta(\theta)}{\int_\Theta\int_\cH R_n(\theta,\eta) d\Pi_\cH(\eta) d\Pi_\Theta(\theta)}
\ee
for every measurable $A \subset \Theta$ and $B \subset \cH$, where $\bD_n = ((x_i, Y_i): 1\leq i\leq n)$.
With a slight abuse of notation, if there is no confusion, $\Pi(\cdot|\bD_n)$ is sometimes used to denote the marginal posterior distribution of $\theta$.
As mentioned in the previous subsection, the denominator of \eqref{eq:posterior} can be bounded using \eqref{eq:usual-lan}.
Also, the expectation of the numerator can be bounded by either the prior probability in the set $A\times B$ (Theorem \ref{thm:dimension}) or by constructing a certain sequence of tests (Theorem \ref{thm:d_n-consistency}).

\begin{theorem}[Dimension]\label{thm:dimension}
Assume that prior conditions \eqref{eq:pi_p_condition}-\eqref{eq:lambda_condition} hold, $\lambda \|\theta_0\|_1 = O(s_0 \log p)$, and $s_0 \log p = o(n)$.
Then there exists a constant $K_{\rm dim} > 1$ such that
\be \label{eq:dimension}
	\E \Pi\big(s_\theta > K_{\rm dim} \big\{s_0 \vee (\log n)^2 \big\} \;\big|\; \bD_n \big) = o(1).
\ee
\end{theorem}

Compared with Theorem 1 of \cite{castillo2015bayesian}, Theorem \ref{thm:dimension} requires two more conditions: $s_0 \log p = o(n)$ and $\lambda \|\theta_0\|_1 = O(s_0 \log p)$.
The former is required for \eqref{eq:usual-lan} to hold.
The latter condition roughly implies that a heavy tail prior is preferred when $\|\theta_0\|_1$ is large.
The additional $(\log n)^2$ term in \eqref{eq:dimension} comes from the prior concentration rate of $\Pi_\cH$ around a KL neighborhood of $\eta_0$.
In the remainder of this section, we let $s_n = 2K_{\rm dim}\{ s_0 \vee (\log n)^2 \}$.
Define the mean Hellinger distance $d_n$ as
$$
	d_n^2 ((\theta_1, \eta_1), (\theta_2, \eta_2)) = \frac{1}{n} \sum_{i=1}^n
	d_H^2( p_{\theta_1,\eta_1,i}, p_{\theta_2,\eta_2,i}),
$$
where $p_{\theta,\eta,i}(y) = \eta(y - x_i^T\theta)$.
For independent observations, this metric is very useful to study asymptotic behavior of the posterior distribution because it is always possible to construct an exponentially consistent sequence of tests.
See \cite{ghosal2007convergence} and references therein.

\begin{theorem}[Consistency in $d_n$] \label{thm:d_n-consistency}
Suppose that conditions given in Theorem \ref{thm:dimension} hold, 
if furthermore, $1 / \phi(s_n) \leq p$ and $s_n \log p = o(n)$, then
$$
	\E \Pi\bigg( d_n((\theta,\eta), (\theta_0, \eta_0)) > K_{\rm Hel} \sqrt{\frac{s_n \log p}{n}} \;\Big|\; \bD_n \bigg) = o(1)
$$
for some constant $K_{\rm Hel} > 0$.
\end{theorem}

\begin{corollary}[$\eta$-consistency] \label{cor:eta-consistency}
Under the conditions of Theorem \ref{thm:d_n-consistency}, it holds that
\be \label{eq:eta-rate}
	\E \Pi\bigg( d_H(\eta, \eta_0) > K_{\rm eta} \sqrt{\frac{s_n \log p}{n}} \;\Big|\; \bD_n \bigg) = o(1)
\ee
for some constant $K_{\rm eta} > 0$.
\end{corollary}

\begin{corollary}[$\theta$-consistency] \label{cor:theta-rate}
Suppose that conditions given in Theorem \ref{thm:d_n-consistency} hold and that $s_n^2 \log p / \phi^2(s_n) = o(n)$.
Then, 
\be\begin{split} \label{eq:theta-rate}
	\E \Pi\bigg(\|\theta-\theta_0\|_1 > K_{\rm theta} \frac{s_n}{\phi(s_n)} \sqrt{\frac{\log p}{n}} \;\Big|\; \bD_n \bigg) = o(1)
	\\
	\E \Pi\bigg(\|\theta-\theta_0\|_2 > K_{\rm theta} \frac{1}{\psi(s_n)} \sqrt{\frac{s_n \log p}{n}} \;\Big|\; \bD_n \bigg) = o(1)
	\\
	\E \Pi\bigg(\|\bX(\theta-\theta_0)\|_2 > K_{\rm theta} \sqrt{s_n \log p} \;\Big|\; \bD_n \bigg) = o(1)
\end{split}\ee
for some constant $K_{\rm theta} > 0$.
\end{corollary}

When $p$ is fixed, a well-known posterior convergence rate $\epsilon_n$ of $\eta$ satisfies $n\epsilon_n^2 \asymp (\log n)^3$ (see \cite{ghosal2001entropies, ghosal2007convergence, chae2016semiparametric}), which agrees with the result of Corollary \ref{cor:eta-consistency}.
If $s_0 \gg (\log n)^2$, the posterior distribution of $\eta$ cannot contract at this rate.
One reason is that prior concentration rate on a KL neighborhood of the true parameter $(\theta_0, \eta_0)$ decreases as $s_0$ increases.
In another viewpoint, the KL divergence of the misspecified model $\eta\mapsto P^{(n)}_{\theta, \eta}$ is not maximized at $\eta_0$ unless $\theta = \theta_0$,
creating an asymptotic bias in estimating $\eta$.

In contrast, as noted in \cite{chae2016semiparametric}, the KL divergence of the misspecified model $\theta\mapsto P_{\theta, \eta}^{(n)}$ is always uniquely maximized at $\theta_0$ provided that $d_H(\eta, \eta_0)$ is sufficiently small.
This means that the recovery rate of $\theta$ is not affected by $\eta$, while the rate of $\eta$ is.
The rate assured by Corollary \ref{cor:theta-rate} is optimal up to the factor $\log p$ provided that $s_0 \gtrsim (\log n)^2$.
This rate agrees with results for the Lasso \cite{buhlmann2011statistics} and for parametric Bayes sparse regression \cite{castillo2015bayesian}.
If $s_0 \lesssim (\log n)^2$, there is an additional $(\log n)^2$ term caused by the unknown error density $\eta$.
As shown in the next subsection, under a slightly stronger condition that $\psi(s_n) \gtrsim 1$, the model dimension can be significantly reduced, which also improves the rate $s_n$ for $s_0 \lesssim (\log n)^2$.

\subsection{Bernstein-von Mises theorem and selection consistency} \label{ssec:bvm}

As noted in \cite{castillo2015bayesian}, even when errors are not normally distributed, Theorem \ref{thm:dimension} and Corollary \ref{cor:theta-rate} can be obtained for  misspecified Gaussian models.
In this subsection, we focus on the asymptotic shape of the marginal posterior distribution of $\theta$.
More specifically, it is shown that the conditional posterior distribution of $\sqrt{n}(\theta_S-\theta_{0,S})$ given that $S_\theta = S$ is asymptotically normal centered on $\Delta_{n,S}$ with the efficient information matrix as variance, where $\Delta_{n,S}$ is the linear estimator with the efficient influence function.
As a consequence, if the true model is consistently selected, a Bayes estimator achieves asymptotic efficiency.
Furthermore, credible sets for $\theta$ provide valid confidence in a frequentist sense.
This assertion holds under the condition that the semi-parametric bias is negligible, which is technically challenging to show in semi-parametric BVM contexts \cite{castillo2015bernstein}.
In our problem, a sufficient condition is that $(s_0 \log p)^6 = O(n^{1-\xi})$ for some $\xi > 0$.

Let
\bean
	\cS_n &=& \left\{S : |S| \leq s_n/2, \; \|\theta_{0,S^c}\|_2 \leq \frac{K_{\rm theta}}{\psi(s_n)} \sqrt{\frac{s_n \log p}{n}} \right\},
	\\
	\cH_n &=& \left\{ \eta\in\cH_0: d_H(\eta,\eta_0) \leq K_{\rm eta} \sqrt{\frac{s_n \log p}{n}}\right\},
\eean
and $\Theta_n$ be the set of every $\theta\in\Theta$ such that $S_\theta \in \cS_n$ and
$\|\theta-\theta_0\|_1$, $\|\theta-\theta_0\|_2$ and $\|\bX(\theta-\theta_0)\|_2$ are bounded by quantities given in \eqref{eq:theta-rate}.
Then, Corollaries \ref{cor:eta-consistency} and \ref{cor:theta-rate} imply that $\Pi(\Theta_n\times\cH_n | \bD_n) \rightarrow 1$ in $P_0^{(n)}$-probability.
Let $G_{n,\eta,S}$ be the $|S|$-dimensional projection of the random vector $\bbG_n \score_{\theta_0, \eta}$ onto $\bbR^{|S|}$, $V_{n,\eta,S}= v_{\eta} \Sigma_S$, and $\cN_{n,\eta,S}$ be the multivariate normal distribution with mean $V_{n,\eta,S}^{-1}G_{n,\eta,S}$ and variance $V_{n,\eta,S}^{-1}$.
For notational convenience, we denote $G_{n,\eta_0,S}, V_{n,\eta_0,S}$ and $\cN_{n,\eta_0,S}$ as $G_{n,S}, V_{n,S}$ and $\cN_{n,S}$, respectively.
Let $\dot L_{n,\eta} = (\score_{\eta} (Y_i - x_i^T \theta_0))_{i=1}^n$ and $\bH_S = \bX_S (\bX_S^T \bX_S)^{-1} \bX_S^T$ be the hat matrix for the model $S$.
The following lemma is useful to characterize the marginal posterior distribution of $\theta$.

\begin{lemma} \label{lem:normal-concentration}
Let $(M_n)$ be any sequence such that $M_n \rightarrow \infty$ and $A_S = \{h\in \bbR^{|S|}: \|h\|_1 > M_n s_n \sqrt{\log p}\}$.
Then,
\be \label{eq:normal-tail}
	\sup_{S\in\cS_n} \sup_{\eta\in\cH_n} \frac{\int_{A_S} \exp\left( h^T G_{n,\eta,S} - \half h^T V_{n,\eta,S}h\right) dh}
	{\int_{\bbR^{|S|}} \exp\left( h^T G_{n,\eta,S} - \half h^T V_{n,\eta,S}h\right) dh} = o_{P_0}(1)
\ee
provided that $s_n\log p = o(n)$ and $\psi(s_n)$ is bounded away from zero.
\end{lemma}

Denote the centered and scaled coefficients as $h = \sqrt{n}(\theta-\theta_0)$ and $h_S = \sqrt{n}(\theta_S-\theta_{0,S})$.
By Lemma \ref{lem:normal-concentration}, it holds for every measurable $B \subset \bbR^{|S|}$ that
$$
	\frac{\int_B \exp\left( h^T G_{n,\eta,S} - \half h^T V_{n,\eta,S}h\right) dh}
	{\int_{\bbR^{|S|}} \exp\left( h^T G_{n,\eta,S} - \half h^T V_{n,\eta,S}h\right) dh} 
	\asymp \frac{\int_B \exp\left( h^T G_{n,\eta,S} - \half h^T V_{n,\eta,S}h\right) dh}
	{\int_{A_S^c} \exp\left( h^T G_{n,\eta,S} - \half h^T V_{n,\eta,S}h\right) dh},
$$
so the density proportional to $e^{h^T G_{n,\eta,S} - \half h^T V_{n,\eta,S}h} 1_{A_S^c}$ is approximately that of the normal distribution $N_{n,\eta,S}$.
Also, $\lambda s_n \sqrt{\log p / n} = o(1)$ implies that 
$$
	\sup_{\theta\in\Theta_n} \left|\log \frac{g_{S_\theta}(\theta_{S_\theta})}{g_{S_\theta}(\theta_{0,S_\theta})} \right| = o(1),
$$
which roughly means that the effect of the prior $g_S$ vanishes as $n$ increases.
As a consequence, if the remainder term of the misspecified LAN expansion is of order $o_{P_0}(1)$, then the total variation distance between $\cN_{n,\eta,S}$ and the conditional posterior distribution of $h_S$ given $\eta$ and $S_\theta=S$ converges to zero in probability.
A sufficient condition for this is that $s_n^5 (\log p)^3 = o(n)$ by Theorem \ref{thm:mLAN}.
Note that it is shown in \cite{chae2016semiparametric} that $\sup_{\eta\in\cH_n}|G_{n,\eta,S_n} - G_{n,\eta_0,S_n}| = o_{P_0}(1)$ and $\sup_{\eta\in\cH_n} |v_\eta - v_{\eta_0}| = o(1)$ for every nonrandom sequence of models $(S_n)$ with $|S_n| = O(1)$.
It follows that $d_V(\cN_{n,\eta,S_n}, \cN_{n,\eta_0,S_n}) = o_{P_0}(1)$.
However, $d_V(\cN_{n,\eta,S}, \cN_{n,\eta_0,S})$, for large $|S|$, may not be close to zero due to the asymptotic biases $G_{n,\eta,S} - G_{n,\eta_0,S}$ and $V_{n,\eta,S}-V_{n,\eta_0,S}$, whose sizes are roughly proportional to the dimension $|S|$.
These biases vanish if $(s_n \log p)^6 = O(n^{1-\xi})$ for some $\xi > 0$ as in the following refined version of the misspecified LAN.

\begin{theorem} \label{thm:LAN}
Let
\bean
	r_n(\theta,\eta) = L_n(\theta,\eta) - L_n(\theta_0,\eta)
	- \sqrt{n}(\theta-\theta_0)^T \bbG_n \score_{\theta_0,\eta_0}
	+ \frac{n}{2} (\theta-\theta_0)^T V_{n,\eta_0}(\theta-\theta_0).
\eean
Then, there exists a sequence $M_n \rightarrow \infty$ such that
\bean
	\E \bigg( \sup_{\theta\in M_n\Theta_n} \sup_{\eta\in\cH_n} |r_n(\theta,\eta)| \bigg)
	= o(1)
\eean
provided that $(s_n\log p)^6  = O(n^{1-\xi})$ for some $\xi > 0$.
\end{theorem}

Let $w_S$ be the posterior probability of a model $S$ given as
\bean
	w_S &\propto&
	\frac{\pi_p(|S|)}{\binom{p}{|S|}} \int\int \exp\left\{ L_n(\widetilde\theta_S,\eta) - L_n(\theta_0,\eta_0) \right\} d\Pi_\cH(\eta) g_S(\theta_S) d\theta_S
\eean
for every $S \subset \{1, \ldots, p\}$.
The marginal posterior distribution of $\theta$ can be expressed as a mixture form
$$
	d\Pi(\theta | \bD_n) = \sum_{S \subset {\{1, \ldots, p\}}} w_S dQ_S(\theta_S) d\delta_0(\theta_{S^c}),
$$
where for every measurable $B \subset \bbR^{|S|}$,
$$
	Q_S(B) = \frac{\int_B\int \exp\left\{ L_n(\widetilde\theta_S,\eta) - L_n(\theta_0,\eta_0) \right\} d\Pi_\cH(\eta) g_S(\theta_S) d\theta_S}
	{\int\int \exp\left\{ L_n(\widetilde\theta_S,\eta) - L_n(\theta_0,\eta_0) \right\} d\Pi_\cH(\eta) g_S(\theta_S) d\theta_S}.
$$
Each mixture component $Q_S$ can be approximated by a normal distribution by the semi-parametric BVM theorem.
Let $\Pi^\infty$ be the probability measure on $\bbR^p$ defined as
$$
	d\Pi^\infty(\theta|\bD_n) = \sum_{S\subset {\{1, \ldots, p\}}} n^{-|S|/2} w_S d\cN_{n,S}(h_S) d\delta_0(\theta_{S^c}),
$$
where $n^{-|S|/2}$ is the determinant of the Jacobian matrix.

\begin{theorem}[Bernstein-von Mises] \label{thm:BvM}
Suppose that conditions given in Theorem \ref{thm:dimension} hold, $(s_n \log p)^6 = O(n^{1-\xi})$ for some $\xi > 0$, $\lambda s_n \sqrt{\log p / n} = o(1)$ and $\psi(s_n)$ is bounded away from 0.
Then,
\be\label{eq:bvm2}
	d_V(\Pi(\cdot|\bD_n), \Pi^\infty(\cdot|\bD_n)) = o_{P_0}(1).
\ee
\end{theorem}

Since posterior mass concentrates on $\Theta_n \times \cH_n$, if $\lambda s_n \sqrt{\log p / n} = o(1)$ and $S\supset S_0$, then $w_S$ can be approximated as
\be\begin{split}\label{eq:w-hat-def}
	\hat w_S &\propto \frac{\pi_p(|S|)}{\binom{p}{|S|}} \int \exp\left\{ \sqrt{n}(\theta_S-\theta_{0,S})^T G_{n,S} - \hbox{$n \over 2$} (\theta_S-\theta_{0,S})^T V_{n,S} (\theta_S-\theta_{0,S})\right\} g_S(\theta_{0,S}) d\theta_S
	\\
	&\propto \frac{\pi_p(|S|)}{\binom{p}{|S|}} \left(\frac{\lambda}{2}\right)^{|S|} \left(\frac{2\pi}{v_{\eta_0}}\right)^{|S|/2} |\bX_S^T \bX_S|^{-1/2} \exp \left( \hbox{$1\over 2 v_{\eta_0}$} \|\bH_S \dot L_{n,\eta_0}\|_2^2\right)
\end{split}\ee
by the LAN and Lemma \ref{lem:normal-concentration}.

\begin{theorem}[Selection] \label{thm:selection}
Suppose that conditions given in Theorem \ref{thm:dimension} hold.
Also, assume that $(s_n\log p)^6  = O(n^{1-\xi})$ for some $\xi > 0$, $\lambda s_n \sqrt{\log p /n} = o(1)$, and $\psi(s_n)$ is bounded away from 0.
Then, there exists a constant $K_{\rm sel}$, depending only on $\eta_0$, such that
$$
	\E \Pi(S_\theta \supsetneq S_0 | \bD_n) \rightarrow 0
$$
provided that $A_4 > K_{\rm sel}$.
\end{theorem}

Since small coefficients cannot be selected by any method, for selection consistency, we need the so-called \emph{beta-min} condition in the following form:
\be\label{eq:beta-min}
	\min \left\{|\theta_{0,i}|: \theta_{0,i} \neq 0 \right\} > \frac{K_{\rm theta}}{\psi(s_n)} \sqrt{\frac{s_n \log p}{n}}.
\ee
Note that under the beta-min condition \eqref{eq:beta-min}, $\cS_n$ contains no strict subset of $S_0$, so combined with Theorem \ref{thm:selection}, it holds that $\Pi(S_\theta=S_0|\bD_n) \rightarrow 1$ in $P_0^{(n)}$-probability.

\section{Proofs}
\label{sec-proofs}

\begin{lemma} \label{lem:secscore-empirical}
Let $(s_n)$ be a sequence of positive integers,
$\Theta_n = \{\theta\in\bbR^p: s_\theta \leq s_n, \|\theta-\theta_0\|_1 \leq 1\}$,
and $f_{\theta,\bar\theta, \eta} = (\theta-\theta_0)^T \secscore_{\bar\theta,\eta}(\theta-\theta_0)$.
Then, it holds that
\bean
	\E \bigg[ \sup_{\theta,\bar\theta \in\Theta_n} \sup_{\eta\in\cH_0} \frac{1}{\sqrt{n}}
	\Big| \bbG_n f_{\theta,\bar\theta, \eta} \Big| \bigg] \lesssim \delta_n,
\eean
where $\delta_n = \sqrt{s_n\log p / n}$.
\end{lemma}
\begin{proof}
Without loss of generality, we may assume that $\theta_0 = 0$.
Let 
$$
	\cF_n = \Big\{f_{\theta,\bar\theta,\eta}: \theta,\bar\theta\in\Theta_n, \eta\in\cH_0 \Big\}.
$$
We first find a bound of bracketing number and envelop function of $\cF_n$, and apply Corollary \ref{cor:maximal_bracket}.
Note that $f_{\theta,\bar\theta,\eta}(x,y) = |x^T\theta|^2 \secscore_\eta(y-x^T\bar\theta)$ and $|x_i^T \theta| \leq \|x_i\|_\infty \|\theta\|_1 \leq L$ for every $\theta\in\Theta_n$.
The map $(x,y) \mapsto F_n(x,y) = L^2 \sup_{|\mu| \leq L} m_2(y-\mu)$ is an envelop function of $\cF_n$ by \eqref{eq:libschitz}, and $\sup_{n\geq 1}\sup_{x\in[-L,L]^p} \E_x F_n^2 \lesssim 1$.
Therefore, $\|F_n\|_n \lesssim 1$, where $\|\cdot\|_n$ is the norm defined in Corollary \ref{cor:maximal_bracket}.

For $(\theta^j, \bar\theta^j, \eta)\in \Theta_n^2\times\cH_0$, $j=1,2$, 
write $f_{\theta^1,\bar\theta^1,\eta_1} - f_{\theta^2,\bar\theta^2,\eta_2} = f_1 + f_2 + f_3$, where
\bean
	f_1 = f_{\theta^1,\bar\theta^1,\eta_1} - f_{\theta^2,\bar\theta^1,\eta_1}, \quad
	f_2 = f_{\theta^2,\bar\theta^1,\eta_1} - f_{\theta^2,\bar\theta^2,\eta_1}, \quad
	f_3 = f_{\theta^2,\bar\theta^2,\eta_1} - f_{\theta^2,\bar\theta^2,\eta_2}.
\eean
Note that $m_k$'s, defined in \eqref{eq:libschitz}, are of polynomial orders, so $\sup_{|\mu| \leq L} m_k(y + \mu) \lesssim m_k(y)$.
Thus, it can be easily shown that
\be \label{eq:f1-f2-bound}
	|f_1(x,y)| \lesssim \|\theta^1 - \theta^2\|_1 m_2(y), \qquad
	|f_2(x,y)| \lesssim \|\bar\theta^1 - \bar\theta^2\|_1 m_3(y).
\ee

To bound $f_3$, consider the class of functions $\cG^K = \{ \secscore_\eta: \eta\in\cH_0 \}$, where $\secscore_\eta \in \cG^K$ is viewed as a map from $[-K, K]$ to $\bbR$.
For a positive integer $\beta$, let
$$
	H^K(\beta) = \sup_{\eta\in\cH_0} \sup_{0\leq k \leq \beta} \sup_{|y| \leq K} \big| \secscore_\eta^{(k)}(y) \big |,
$$
where $\secscore_\eta^{(k)}$ is the $k$th order derivative of the map $y \mapsto \secscore_\eta(y)$.
Then by Theorem 2.7.1 of \cite{van1996weak},
\be \label{eq:uniform-entropy-bound}
	\log N(\delta, \cG^K, \|\cdot\|_\infty) \leq D_\beta (K+1) \bigg\{ \frac{H^K(\beta)}{\delta} \bigg\}^{1/\beta},
\ee
where $D_\beta$ is a constant depending only on $\beta$.
Note that $\int_K^\infty y^4 e^{-y^2} dy \leq K^3 e^{-K^2}$ for every large enough $K$.
Note also that there exist constants $a_j, 1\leq j \leq 3,$ depending only on $\sigma_1, \sigma_2$ and $M$ such that
$$
	\sup_{\eta\in\cH_0} \big| \secscore_\eta(y) \big| \leq a_1 y^2
	\quad \textrm{and} \quad
	\eta_0(y) \leq a_2 e^{-a_3 y^2}
$$
for every large enough $|y|$.
Thus, for a constant $C_1 >0$, we have
\bean
	\int_{\{y:|y| \geq C_1 \sqrt{\log (1/\delta)}\}} \sup_{\eta\in\cH_0} \big| \secscore_\eta(y) \big|^2 dP_{\eta_0}(y) 
	\leq \int_{\{y:|y| \geq C_1 \sqrt{\log (1/\delta)}\}} a_1^2 a_2 y^4 e^{-a_3 y^2} dy
	\\
	\leq \int_{\{y:|y| \geq C_1 \sqrt{a_3 \log (1/\delta)}\}} \frac{a_1^2 a_2}{a_3^{5/2}} y^4 e^{-y^2} dy
	\leq \frac{a_1^2 a_2}{a_3^{5/2}} \left(C_1 \sqrt{a_3 \log (1/\delta)} \right)^3 \delta^{C_1^2 a_3}
\eean
for every small enough $\delta > 0$.
Therefore, we can choose $C_1 > 0$, depending only on $\sigma_1, \sigma_2$ and $M$, such that
\be\label{eq:nmixture-tail}
	\int_{\{y:|y| \geq C_1 \sqrt{\log (1/\delta)}\}} \sup_{\eta\in\cH_0} \big| \secscore_\eta(y) \big|^2 dP_{\eta_0}(y) \leq \delta^2
\ee
for every small enough $\delta > 0$.
By \eqref{eq:uniform-entropy-bound} and \eqref{eq:nmixture-tail}, for every small enough $\delta > 0$ there exists a partition $\{\cH^l : 1 \leq l \leq N(\delta)\} $ of $\cH_0$ into $N(\delta)$ sets such that 
$$
	\log N(\delta) \leq 2 D_\beta K_\delta \bigg(\frac{H^{K_\delta}(\beta)}{\delta} \bigg)^{1/\beta}
$$
and
\be\label{eq:ddot-L2-bound}
	\int \sup_{\eta_1, \eta_2 \in \cH_l} \sup_{|a| \leq L} \Big| \secscore_{\eta_1}(y-a) - \secscore_{\eta_2}(y-a) \Big|^2 dP_{\eta_0}(y) \lesssim \delta^2 \sqrt{\log(1/\delta)}
\ee
for every $l \leq N(\delta)$, where $K_\delta = C_1 \sqrt{\log(1/\delta)}$.
If $\epsilon = \delta^{1-\gamma}$ for small constant $\gamma > 0$, the right hand side of \eqref{eq:ddot-L2-bound} is bounded by $\epsilon^2$.
Combining with \eqref{eq:f1-f2-bound}, we conclude that there exists a constant $C_2 > 0$ depending only on $\sigma_1, \sigma_2$ and $M$, such that
\be\begin{split} \nonumber
	\log N_{[]}^n(C_2 \epsilon, \cF_n) 
	&\leq \log N(\epsilon^{1/(1-\gamma)}) + 2 \log N(\epsilon, \Theta_n, \|\cdot\|_1) 
	\\
	&\lesssim \log N(\epsilon^{1/(1-\gamma)}) + \log\left\{ \binom{p}{s_n} \left(\frac{1}{\epsilon}\right)^{s_n}\right\}
	\\
	&\leq \log N(\epsilon^{1/(1-\gamma)}) + s_n \log p + s_n \log\Big( \frac{1}{\epsilon}\Big),
\end{split}\ee
where $N_{[]}^n$ is the bracket number defined in Section \ref{sec:maximal}.
Thus, 
\be \begin{split} \label{eq:secscor-empirical-bound}
	\E \Big( \sup_{f\in \cF_n} \big| \bbG_n f \big| \Big)
	&\lesssim \int_0^1 \sqrt{\log N(\{\epsilon/C_2\}^{1/(1-\gamma)})} d\epsilon
	+ \sqrt{s_n \log p}
\end{split} \ee
by Corollary \ref{cor:maximal_bracket}.
Since $H^K(1) \lesssim K^3$, we have
\bean
	\log N(\delta) \lesssim \frac{(\log\delta)^2}{\delta} \leq \delta^{-3/2}
\eean
for small enough $\delta > 0$.
Thus, the integral in \eqref{eq:secscor-empirical-bound} is bounded by a constant multiple of
$$
	\int_0^1 \epsilon^{\frac{-3}{4(1-\gamma)}} d\epsilon
$$
which is finite for small $\gamma$.
Thus, \eqref{eq:secscor-empirical-bound} is bounded by a constant multiple of $\sqrt{s_n \log p}$.
This completes the proof.
\end{proof}

\medskip
{\it Proof of Theorem \ref{thm:mLAN}.}
By the Taylor expansion, $L_n(\theta,\eta) - L_n(\theta_0,\eta)$ is equal to
\bean
	\sqrt{n}(\theta-\theta_0)^T \bbG_n \score_{\theta_0,\eta}
	+ n (\theta-\theta_0)^T \int_0^1 \Big\{ (1-t) \bbP_n \secscore_{\theta(t), \eta} \Big\} dt (\theta-\theta_0),
\eean
where $\theta(t) = \theta_0 + t(\theta-\theta_0)$.
The quadratic term of the Taylor expansion can be decomposed as $Q_{n,1}(\theta,\eta) + Q_{n,2}(\theta,\eta) + Q_{n,3}(\theta,\eta)$, where
\bean
	Q_{n,1}(\theta,\eta) &=& n\int_0^1 (1-t) \frac{1}{\sqrt{n}} \bbG_n (\theta-\theta_0)^T \secscore_{\theta(t), \eta} (\theta-\theta_0) dt
	\\
	Q_{n,2}(\theta,\eta) &=& \int_0^1 (1-t) \sum_{i=1}^n \Big[ (\theta-\theta_0)^T \Big\{ \E \secscore_{\theta(t), \eta}(x_i, Y_i) - \E \secscore_{\theta_0, \eta}(x_i, Y_i) \Big\} (\theta-\theta_0) \Big] dt
	\\
	Q_{n,3}(\theta,\eta) &=& \frac{1}{2} \sum_{i=1}^n (\theta-\theta_0)^T \E \secscore_{\theta_0, \eta}(x_i, Y_i) (\theta-\theta_0).
\eean
Since
$$
	\frac{1}{\sqrt{n}} \bbG_n (\theta-\theta_0)^T \secscore_{\theta(t), \eta} (\theta-\theta_0) = \frac{\|\theta-\theta_0\|_1^2}{\sqrt{n}} \bbG_n \frac{(\theta-\theta_0)^T}{\|\theta-\theta_0\|_1} \secscore_{\theta(t), \eta} \frac{(\theta-\theta_0)}{\|\theta-\theta_0\|_1},
$$
it holds that
$$
	\E \bigg( \sup_{\theta\in\Theta_n} \sup_{\eta\in\cH_0} \big|Q_{n,1}(\theta,\eta)\big| \bigg)
	\lesssim n \epsilon_n^2 \delta_n
$$
by Lemma \ref{lem:secscore-empirical}, where $\delta_n = \sqrt{s_n\log p / n}$.
Each summand in the definition of $Q_{n,2}$ is equal to
$$
	|x_i^T (\theta-\theta_0)|^2 \E \Big\{\secscore_\eta (Y_i - x_i^T \theta(t)) - \secscore_\eta(Y_i - x_i^T \theta_0)\Big\} 
	\lesssim |x_i^T (\theta-\theta_0)|^2 \|\theta-\theta_0\|_1,
$$
so
$$
	\bigg( \sup_{\theta\in\Theta_n} \sup_{\eta\in\cH_0} \big|Q_{n,2}(\theta,\eta)\big| \bigg)
	\lesssim \epsilon_n \sup_{\theta\in\Theta_n} \|X(\theta-\theta_0)\|_2^2.
$$
Since
$$
	Q_{n,3}(\theta,\eta) = -\frac{v_\eta}{2} \|X(\theta-\theta_0)\|_2^2,
$$
the proof is complete.
\qed

\begin{lemma}\label{lem:score-empirical}
It holds that
\be \begin{split}\label{eq:score-empirical-general}
	\E \bigg[ \sup_{\eta\in\cH_0} \Big\| \bbG_n \score_{\theta_0,\eta}\Big\|_\infty \bigg] &\lesssim \sqrt{\log p}.
\end{split}\ee
\end{lemma}
\begin{proof}
Without loss of generality, we may assume that $\theta_0=0$.
Consider the class of real valued functions
$$
	\cF_n = \Big\{e_j^T \score_{\theta_0, \eta}: 1 \leq j \leq p, \eta\in\cH_0\Big\},
$$
where $e_j$ is the $j$th unit vector in $\bbR^p$.
Then, it is obvious that
$$
	\sup_{\eta\in\cH_0} \Big\| \bbG_n \score_{\theta_0,\eta}\Big\|_\infty = \sup_{f \in\cF_n} |\bbG_n f|.
$$
We apply Corollary \ref{cor:maximal_bracket} to bound the right hand side.
Note that $|f(x,y)| \lesssim m_1(y)$ for every $f \in \cF_n$, so there exists an envelop $F_n$ of $\cF_n$ such that $\|F_n\|_n \lesssim 1$, where $\|\cdot\|_n$ is the norm defined in Corollary \ref{cor:maximal_bracket}.
Let $\cG = \{\score_\eta: \eta\in\cH_0\}$.
Then, by applying Corollary 2.7.4 of \cite{van1996weak} with $\alpha=d=1$ and $r=2$, we have that $\log N_{[]} (\epsilon, \cG, L_2(P_{\eta_0})) \lesssim \epsilon^{-1}$.
This implies that $\log N_{[]}^n(\cF_n, \epsilon) \lesssim \epsilon^{-1} + \log p$.
Thus, the proof is complete by Corollary \ref{cor:maximal_bracket}.
\end{proof}

\begin{lemma}\label{lem:denom_bd}
Assume that \eqref{eq:lambda_condition} holds and $s_0 \log p = o(n)$.
Then, there exists a positive constant $D$, depending only on $\sigma_1, \sigma_2$ and $M$, such that the $P_0^{(n)}$-probabilities of the event
\be\label{eq:denom_tot_bd}
	\bigg\{ \int_{\Theta\times\cH} R_n(\theta,\eta) d\Pi(\theta,\eta) 
	\geq  \exp\Big[D \big\{\log \pi_p(s_0) -s_0 \log p -\lambda\|\theta_0\|_1 - (\log n)^3 \big\} \Big] \bigg\}
\ee
converge to 1.
\end{lemma}
\begin{proof}
Let $\epsilon_n = n^{-1/2} (\log n)^{3/2}$ and
$$
	\cH_n = \bigg\{\eta\in\cH_0: -P_{\eta_0} \bigg(\log \frac{\eta}{\eta_0}\bigg)
	 \leq \epsilon_n^2, ~  P_{\eta_0} \bigg(\log \frac{\eta}{\eta_0}\bigg)^2 \leq \epsilon_n^2 \bigg\},
$$
then
\be\begin{split} \label{eq:denom_tot}
	\int_{\Theta\times\cH} R_n(\theta,\eta) d\Pi(\theta,\eta)
	&\geq \int_{\Theta\times\cH_n} R_n(\theta_0,\eta) \frac{R_n(\theta,\eta)}{R_n(\theta_0,\eta)}  d\Pi(\theta,\eta)
	\\
	&\geq \int_{\Theta\times\cH_n} R_n(\theta_0,\eta) \inf_{\eta\in\cH_n}\bigg(\frac{R_n(\theta,\eta)}{R_n(\theta_0,\eta)}\bigg)  d\Pi(\theta,\eta)
	\\
	&\geq \int_{\cH_n} R_n(\theta_0,\eta) d\Pi_\cH(\eta) \times
	\int_\Theta \inf_{\eta\in\cH_0}\bigg(\frac{R_n(\theta,\eta)}{R_n(\theta_0,\eta)}\bigg) d\Pi_\Theta(\theta).
\end{split}\ee
It is shown in \cite{ghosal2001entropies} (see the proof of Theorem 6.2) that
$\log \Pi_\cH (\cH_n) \gtrsim -n\epsilon_n^2$.
By Lemma 8.1 of \cite{ghosal2000convergence},
\be\label{eq:denom_eta}
	P_0^{(n)} \bigg( \bigg\{\int_{\cH_n} R_n(\theta_0,\eta) d\Pi_\cH(\eta) \geq e^{-Cn\epsilon_n^2} \Pi_\cH(\cH_n)\bigg\} \bigg) \rightarrow 1
\ee
for any $C > 1$.

Let $\Theta_n = \{\theta\in\Theta: \sqrt{n} \|\theta-\theta_0\|_1 \leq 1, \; S_\theta = S_0\}$, then
\be\begin{split}\label{eq:denom_bd}
	\int_\Theta \inf_{\eta\in\cH_0}\bigg(\frac{R_n(\theta,\eta)}{R_n(\theta_0,\eta)}\bigg) d\Pi_\Theta(\theta)
	\geq \int_{\Theta_n} \exp\bigg( \inf_{\eta\in\cH_0} \Big\{L_n(\theta,\eta) - L_n(\theta_0,\eta)\Big\} \bigg) d\Pi_\Theta(\theta)
	\\
	\geq \frac{\pi_p(s_0)}{\binom{p}{s_0}} \int_{\Theta_n} \exp\bigg( \inf_{\eta\in\cH_0} \Big\{L_n(\theta,\eta) - L_n(\theta_0,\eta)\Big\} \bigg)
	g_{S_0}(\theta_{S_0}) \;d\theta_{S_0}.
\end{split}\ee
By Theorem \ref{thm:mLAN}, the last exponent of \eqref{eq:denom_bd} is bounded below by
$$
	\inf_{\eta\in\cH_0} \bigg\{ \sqrt{n}(\theta-\theta_0)^T \bbG_n \score_{\theta_0,\eta}
	- \frac{n}{2}(\theta-\theta_0)^T V_{n,\eta}(\theta-\theta_0)\bigg\}
	+ o_{P_0}(1),
$$
where the $o_{P_0}(1)$ term does not depend on $\theta$ and $\eta$.
Note that
$$
	\E \bigg[\sup_{\theta\in\Theta_n} \sup_{\eta\in\cH_0} \bigg| \sqrt{n} (\theta-\theta_0)^T \bbG_n \score_{\theta_0,\eta}\bigg| \bigg] = O(\sqrt{\log p})
$$
by Lemma \ref{lem:score-empirical}, and
$$
	\sup_{\theta\in\Theta_n} \sup_{\eta\in\cH_0} n (\theta-\theta_0)^T V_{n,\eta_0}(\theta-\theta_0) = O(1).
$$
Thus, for every real sequence $M_n \rightarrow \infty$, \eqref{eq:denom_bd} is bounded below by
\be\begin{split} \label{eq:laplace_trans}
	&\frac{\pi_p(s_0)}{\binom{p}{s_0}} e^{-M_n \sqrt{\log p}} \int_{\Theta_n} g_{S_0}(\theta_{S_0}) d\theta_{S_0}
	\\
	&\geq \frac{\pi_p(s_0)}{\binom{p}{s_0}} e^{-M_n \sqrt{\log p} - \lambda \|\theta_0\|_1} \int_{\Theta_n} g_{S_0}(\theta_{S_0} - \theta_{0,S_0}) d\theta_{S_0}
\end{split}\ee
with $P_0^{(n)}$-probability tending to 1.
The last integral of \eqref{eq:laplace_trans} is equal to
\bean
	&&\int_{\Theta_n} \bigg(\frac{\lambda}{2}\bigg)^{s_0} e^{-\lambda \|\theta_{S_0} -\theta_{0,S_0}\|_1} d\theta_{S_0}
	\geq \bigg(\frac{\lambda}{2}\bigg)^{s_0} e^{-\frac{\lambda}{\sqrt{n}}} \int_{\Theta_n}  \;d\theta
	\\
	&&\geq \bigg(\frac{\lambda}{2}\bigg)^{s_0} e^{-\frac{\lambda}{\sqrt{n}}} \int_{\{\theta_{S_0}\in\bbR^{S_0}:\|\theta_{S_0}-\theta_{0,S_0}\|_2 \leq 1/\sqrt{s_0 n}\}}  \;d\theta_{S_0}
	\\
	&&= \bigg(\frac{\lambda}{2}\bigg)^{s_0} e^{-\frac{\lambda}{\sqrt{n}}} \frac{\pi^{s_0/2}}{\Gamma(s_0/2+1)} (ns_0)^{-s_0/2}
	= \bigg(\frac{\lambda}{\sqrt{n}}\bigg)^{s_0}  \bigg(\frac{\sqrt{\pi}}{2\sqrt{s_0}}\bigg)^{s_0} \frac{e^{-\frac{\lambda}{\sqrt{n}}}}{\Gamma(s_0/2+1)}.
\eean
Since $\sqrt{n}/p \leq \lambda \leq \sqrt{n\log p}$, and $\binom{p}{s_0} \leq p^{s_0} / \Gamma(s_0+1)$, \eqref{eq:laplace_trans} is bounded below by
\bean
	\frac{\pi_p(s_0)}{p^{2s_0}} \frac{\Gamma(s_0+1)}{\Gamma(s_0/2+1)}
	\bigg(\frac{\sqrt{\pi}}{2\sqrt{s_0}}\bigg)^{s_0} e^{-\lambda\|\theta_0\|_1 - (M_n+1) \sqrt{\log p}}.
\eean
Since 
\bean
	\frac{\Gamma(s_0+1)}{\Gamma(s_0/2+1)} \bigg(\frac{\sqrt{\pi}}{2}\bigg)^{s_0} \gtrsim 1,
\eean
the last display is bounded below by a constant multiple of
\bean
	\frac{\pi_p(s_0)}{p^{2s_0}} e^{-\lambda\|\theta_0\|_1 - (M_n+1) \sqrt{\log p} - (s_0 \log s_0) /2}.
\eean
Combining with \eqref{eq:denom_tot} and \eqref{eq:denom_eta}, the proof is complete by letting $M_n = \sqrt{\log p}$.
\end{proof}

\medskip
{\it Proof of Theorem \ref{thm:dimension}.}
For $R > s_0$ and $B = \{(\theta,\eta): |S_\theta| \geq R\}$,
\bean
	\Pi(B) = \sum_{s=R}^p \pi_p(s) \leq \sum_{s=R}^p \pi_p (s_0)
	\bigg(\frac{A_2}{p^{A_4}}\bigg)^{s-s_0}
	\leq \pi_p (s_0)  \bigg(\frac{A_2}{p^{A_4}}\bigg)^{R-s_0} 
	\sum_{j=0}^\infty \bigg(\frac{A_2}{p^{A_4}}\bigg)^j.
\eean
Under the condition \eqref{eq:pi_p_condition}, every constant $C_1 > A_3$, $-\log \pi_p(s_0) \leq C_1 s_0 \log p$ for large enough $n$.
Let $E_n$ be the event \eqref{eq:denom_tot_bd}.
Since $\lambda \|\theta_0\|_1 = O(s_0 \log p)$ and $\log \pi_p(s_0) \gtrsim s_0\log p$, we have, by Lemma \ref{lem:denom_bd},
\be\begin{split}\label{eq:dimen_up_bd}
	&\E \Pi(B | \bD_n) 1_{E_n} 
	\\
	&\leq \exp\Big[C_2 \big\{s_0 \log p +\lambda\|\theta_0\|_1 + (\log n)^3 - \log \pi_p(s_0) \big\} \Big]
	\E \int_B R_n(\theta,\eta) d\Pi(\theta,\eta)
	\\	
	&\leq \Pi(B) \exp\Big[C_3 \big\{s_0 \log p + (\log n)^3 \big\}  \Big]
\end{split}\ee
for some constants $C_2$ and $C_3$.
If $(\log n)^3 \leq s_0 \log p$, then, for $R = K_1 s_0$ with sufficiently large constant $K_1 >0$, the right hand side of \eqref{eq:dimen_up_bd} converges to 0.
Otherwise, for $R = K_2 (\log n)^2$ with sufficiently large constant $K_2 >0$, the right hand side of \eqref{eq:dimen_up_bd} converges to 0.
Since $\E \Pi(B | \bD_n) = \E \Pi(B | \bD_n)1_{E_n} + P_0^{(n)}(E_n^c)$ and $P_0^{(n)}(E_n^c) = o(1)$, the proof is complete.
\qed

\medskip
{\it Proof of Theorem \ref{thm:d_n-consistency}.}
We first prove that there exists a constant $C_1 > 0$, depending only on $\sigma_1, \sigma_2$ and $M$, such that
$\E \Pi(\theta \notin \Theta_n | \bD_n ) = o(1)$, where 
\be\label{eq:Theta_n-def-hel}
	\Theta_n = \Big\{\theta: s_\theta \leq s_n/2, \; \|\theta-\theta_0\|_1 \leq \frac{C_1 \sqrt{s_n}}{\phi(s_n)} \Big\}.
\ee
Note first that $\E \Pi(s_\theta > s_n/2 | \bD_n) = o(1)$ by Theorem \ref{thm:dimension}.
Since every $\eta \in \cH_0$ has a sub-Gaussian tail, there exist positive constants $a$ and $b$, depending only on $\sigma_1, \sigma_2$ and $M$, such that
$$
	\sup_{\eta\in\cH_0} P_\eta \Big( e^{\epsilon_i^2 /a} -1 - \frac{\epsilon_i^2}{a} \Big) a^2 \leq \frac{b}{2},
$$
where $\epsilon_i$'s are \iid\ following $P_\eta$.
By Lemma \ref{lem:bernstein-inequality},
\bean
	\sup_{\eta\in\cH_0} P_\eta\Big(\Big| \frac{1}{n} \sum_{i=1}^n \epsilon_i^2 - \sigma_\eta^2 \Big| > t \Big) \leq 2e^{\frac{-nt^2}{2(b+at)}},
\eean
where $\sigma_\eta^2$ is the variance of $P_\eta$.
Define a sequence of tests $(\phi_n)$ as $\phi_n = 1_{\{\|\bY- \bX \theta_0\|_2^2/n - \sigma_{\eta_0}^2 > t\}}$.
Then,
$$
	P_0 \phi_n \leq 2e^{\frac{-nt^2}{2(b+at)}}.
$$
Also, for $\theta$ with $s_\theta \leq s_n/2$,
\bean
	P_{\theta,\eta} (1-\phi_n) &=& P_{\theta,\eta} \bigg( \frac{\|\bY - \bX\theta_0\|_2^2}{n} - \sigma_{\eta_0}^2 \leq t\bigg)
	\\
	&\leq& P_{\theta,\eta} \bigg( \frac{\|\bX(\theta-\theta_0)\|_2^2}{2n} - \frac{\|\bY - \bX\theta\|_2^2}{n} - \sigma_{\eta_0}^2 \leq t\bigg)
	\\
	&\leq& P_{\theta,\eta} \bigg( \phi^2(s_n) \frac{\|\theta-\theta_0\|_1^2}{2 s_n} \leq \frac{\|\bY - \bX\theta\|_2^2}{n} + \sigma_{\eta_0}^2 + t\bigg).
\eean
By taking $t=1$, this implies that there exist constants $C_1, C_2 > 0$, depending only on $\sigma_1, \sigma_2$ and $M$, such that
$$
	\sup_{\phi(s_n) \|\theta-\theta_0\|_1 \geq C_1\sqrt{s_n}} \sup_{\eta\in\cH_0} P_{\theta,\eta} (1-\phi_n) \leq 2 e^{- C_2 n}.
$$
Define $\Theta_n$ as \eqref{eq:Theta_n-def-hel}, and let $E_n$ be the event \eqref{eq:denom_tot_bd}.
Then by Lemma \ref{lem:denom_bd},
\be\begin{split} \label{eq:consistency-technique}
	&\E \Pi(\theta\notin\Theta_n | \bD_n) = \E \Pi(\theta\notin\Theta_n | \bD_n) 1_{E_n} (1- \phi_n) + o(1)
	\\
	&\leq \sup_{\theta\notin\Theta_n} \sup_{\eta\in\cH_0} P_{\theta,\eta}(1-\phi_n) \exp \Big[ C_3 \big\{ s_0\log p + (\log n)^3 \big\} \Big] + o(1),
\end{split}\ee
where $C_3 > 0$ is a constant.
Since $s_0 \log p + (\log n)^3 = o(n)$, the last display is of order $o(1)$.

Next, it is easy to see that
$$
	d_n((\theta^1, \eta_1), (\theta^2, \eta_2)) \lesssim \|\theta^1-\theta^2\|_1 + d_H(\eta_1, \eta_2).
$$
Since $\log N(\epsilon, \cH_0, d_H) \lesssim \{\log(1/\epsilon)\}^3$ (see Theorem 3.3 of \cite{ghosal2001entropies}), we have that $\log N( \epsilon/36, \Theta_n \times \cH_0, d_n) \lesssim \log N(\epsilon)$, where
\bean
	\log N(\epsilon) &=&  s_n \bigg\{ \log p + \log\Big(\frac{1}{\epsilon}\Big) + \log \Big( \frac{1}{\phi(s_n)}\Big)\bigg\}
	+ \bigg\{ \log\Big( \frac{1}{\epsilon} \Big)\bigg\}^3
	\\
	&\lesssim& s_n \log p + s_n \log\Big(\frac{1}{\epsilon}\Big)
	+ \bigg\{ \log\Big( \frac{1}{\epsilon} \Big)\bigg\}^3.
\eean
By Lemmas 2 and 9 of \cite{ghosal2007convergence}, for every $\epsilon > 0$ with $e^{-n\epsilon^2/2} \leq 1/2$, there exist tests $\varphi_n$ such that for some constant $C_4 > 0$
$$
	P_0^{(n)} \varphi_n \leq 2\exp\Big[ C_4 \log N(\epsilon) - \half n\epsilon^2 \Big]
$$
and $P_{\theta,\eta}^{(n)}(1-\varphi_n) \leq e^{-n\epsilon^2/2}$ for all $(\theta, \eta)\in\Theta_n\times \cH_0$ such that $d_n((\theta, \eta), (\theta_0, \eta_0)) > \epsilon$.
Let $\epsilon_n = C_5\sqrt{s_n \log p / n}$ for large enough constant $C_5$, then
similarly to \eqref{eq:consistency-technique}, we have that
$$
	\E \Pi\Big(\theta \in \Theta_n: d_n((\theta, \eta), (\theta_0, \eta_0)) > \epsilon_n \;\Big|\; \bD_n \Big) = o(1),
$$
so the proof is complete.
\qed

\medskip
{\it Proof of Corollary \ref{cor:eta-consistency}.}
Assume that there exist constants $D>0$ and $\delta > 0$ such that
\be \label{eq:inf-eta}
	\inf_{y\in\bbR} d_H(\eta_0, T_y(\eta)) \geq D(d_H(\eta_0,\eta) \wedge \delta)
\ee
for every $\eta\in\cH_0$, where $(T_y(\eta))(x) = \eta(x+y)$.
Then, $d_n((\theta, \eta), (\theta_0, \eta_0)) \leq \epsilon_n$ for some $\epsilon_n = o(1)$ implies that $d_H(\eta,\eta_0) < \epsilon_n / D$.
Therefore, by Theorem \ref{thm:d_n-consistency}, it suffices to prove \eqref{eq:inf-eta}.

Note that 
\bean
	h^2(\eta, \eta(\cdot+y))
	&=& \int \Big(\sqrt{\eta(x+y)} - \sqrt{\eta(x)}\Big)^2 dx
	= y_i^2 \int \bigg(\int_0^1 \frac{\eta^\prime(x+ty)}
	{\sqrt{\eta(x+ty_i)}}dt\bigg)^2 dx 
	\\
	&\leq& y^2 \int \int_0^1 \bigg(\frac{\eta^\prime(x+ty)}
	{\eta(x+ty)}\bigg)^2 \eta(x+ty) \;dt\,dx 
	\leq C^2 y^2
\eean
for some $C > 0$ and every $y\in\bbR$, where $\eta^\prime$ is the derivative of $\eta$.
If $|y| \leq d_H(\eta_0,\eta) / (2C)$, then
\bean
	d_H(\eta_0, T_y(\eta)) &\geq& d_H(\eta_0,\eta) - d_H(\eta,T_y(\eta))
	\\
	&\geq& d_H(\eta_0,\eta) - C|y| \geq \half d_H(\eta_0,\eta).
\eean
If $|y| > d_H(\eta_0,\eta) / (2C)$
\bean
	d_H(\eta_0, T_y(\eta)) \geq d_V(\eta_0, T_y(\eta)) \geq 2 \int_0^{|y|} \eta_0(x) dx \gtrsim |y|\wedge \delta
\eean
for some $\delta > 0$, where the last inequality holds by continuity and positivity of $\eta_0$ at the origin.
\qed

\medskip
{\it Proof of Corollary \ref{cor:theta-rate}.}
Let $\epsilon_n = K_{\rm Hel} \sqrt{s_n \log p/n}$ and 
$$
	\Theta_n = \Big\{\theta\in\Theta: s_\theta \leq s_n/2,\;
	d_n ((\theta, \eta), (\theta_0, \eta_0)) \leq \epsilon_n \Big\}.
$$
Then, by Theorems \ref{thm:dimension} and \ref{thm:d_n-consistency}, $\bbE \Pi(\theta\in\Theta_n | \bD_n) \rightarrow 1$.
Note that there exist positive constants $C_1$ and $\delta$ depending only on $\sigma_1, \sigma_2$ and $M$ (see \cite{chae2016semiparametric}) such that
$$
	d_H^2 (p_{\theta, \eta, i}, p_{\theta_0, \eta_0, i}) \geq C_1^2 \Big( |x_i^T (\theta-\theta_0)| \wedge \delta \Big)^2
$$
for every $\theta \in \Theta$.
For $\theta\in\Theta_n$, and let $\bbN_{\delta, n} = \{ i\leq n: |x_i^T (\theta - \theta_0)| \geq \delta\}$ and $N_{\delta, n}$ be the cardinality of $\bbN_{\delta,n}$.
Then,
\be\begin{split} \label{eq:tech-ineq}
	\epsilon_n^2 &\geq d_n^2 ((\theta, \eta), (\theta_0, \eta_0)) 
	\geq \frac{C_1^2}{n} \sum_{i=1}^n \Big( |x_i^T (\theta-\theta_0)| \wedge \delta \Big)^2
	\\
	&\geq C_1^2 \delta^2 \frac{N_{\delta,n}}{n} + \frac{C_1^2}{n} \sum_{i\notin \bbN_{\delta,n}} |x_i^T (\theta-\theta_0)|^2,
\end{split}\ee
so we have that $N_{\delta,n}/n \leq \epsilon_n^2 / (C_1^2 \delta^2)$.
Since
\bean
	\sum_{i\notin \bbN_{\delta,n}} |x_i^T (\theta-\theta_0)|^2
	&\geq& \sum_{i=1}^n |x_i^T (\theta-\theta_0)|^2 - N_{\delta,n} \max_{i\geq 1} |x_i^T (\theta-\theta_0)|^2
	\\
	&\geq& \phi^2(s_n) \frac{n}{s_n} \|\theta-\theta_0\|_1^2 - L^2 N_{\delta, n} \|\theta - \theta_0 \|_1^2,
\eean
we have that
\bean
	\epsilon_n^2 \geq C_1^2 \|\theta - \theta_0 \|_1^2 \bigg( \frac{\phi^2(s_n)}{s_n} - \frac{L^2 N_{\delta,n}}{n}\bigg)
	\geq C_1^2 \|\theta - \theta_0 \|_1^2 \bigg( \frac{\phi^2(s_n)}{s_n} - \frac{L^2 \epsilon_n^2}{C_1^2 \delta^2}\bigg).
\eean
Since $s_n^2 \log p / \phi^2(s_n) = o(n)$, the last display is bounded below by
$C_1^2 \|\theta-\theta_0\|_1^2 \phi^2(s_n) / (2s_n)$ for large enough $n$.
Therefore,
$$
	\|\theta-\theta_0\|_1^2 \leq \frac{2K_{\rm Hel}^2}{C_1^2} \frac{s_n^2 \log p}{\phi^2(s_n) n}.
$$

From the first line of \eqref{eq:tech-ineq}, we have that
$$
	\frac{n\epsilon_n^2}{C_1^2} \geq \|\bX(\theta-\theta_0)\|_2^2 - \sum_{i \in \bbN_{\delta,n}} \Big(|x_i^T(\theta-\theta_0)|^2 - \delta^2\Big).
$$
Therefore, 
\bean
	\|\bX(\theta-\theta_0)\|_2^2 &\leq& \frac{n\epsilon_n^2}{C_1^2} + \sum_{i \in \bbN_{\delta,n}} |x_i^T(\theta-\theta_0)|^2
	\leq \frac{n\epsilon_n^2}{C_1^2} + L^2 N_{\delta,n} \|\theta-\theta_0\|_1^2
	\\
	&\leq& \frac{K_{\rm Hel}^2}{C_1^2} \Big( 1 + L^2 \|\theta-\theta_0\|_1^2 / \delta^2 \Big) s_n \log p.
\eean
Since $\|\theta - \theta_0\|_1 = o(1)$ by the first assertion of \eqref{eq:theta-rate}, it holds that $\|\bX(\theta-\theta_0)\|_2 \lesssim K_{\rm Hel}^2 \sqrt{s_n \log p}$.
Also, by the definition of $\psi(s)$, we conclude that
$\psi(s_n) \|\theta-\theta_0\|_2 \lesssim K_{\rm Hel}^2 \sqrt{s_n \log p / n}$.
\qed

\medskip
{\it Proof of Lemma \ref{lem:normal-concentration}.}
Note that $|h^T G_{n,\eta,S}| \leq \|h\|_1 \|G_{n,\eta,S}\|_\infty$ and
$$
	\bbE \left(\sup_{S\in\cS_n}\sup_{\eta\in\cH_n} \| G_{n,\eta,S}\|_\infty\right)
	\lesssim \sqrt{\log p}
$$
by Lemma \ref{lem:score-empirical}.
Since $h^T V_{n,\eta,S} h \geq v_\eta \phi^2(s_n) \|h\|_1^2 / s_n$, $\phi(s_n) \geq \psi(s_n)$ , $\sup_{\eta\in\cH_n} |v_\eta - v_{\eta_0}| = o(1)$ (it is shown in \cite{chae2016semiparametric} that $\lim_{d_H(\eta,\eta_0)\rightarrow 0} |v_\eta - v_{\eta_0}| = 0$), $v_{\eta_0} \gtrsim 1$ and $M_n \rightarrow \infty$, we have that
$$
	\sup_{S\in\cS_n} \sup_{h \in A_S} \sup_{\eta\in\cH_n}  \frac{|h^T G_{n,\eta,S}|}{h^T V_{n,\eta,S} h} = o_{P_0}(1).
$$
Also, $h^T V_{n,\eta,S} h \geq v_{\eta} \psi^2(s_n) \|h\|_2^2$ implies that there exist events $(\Omega_n)$ and a constant $C >0$ such that $P_0^{(n)}(\Omega_n) \rightarrow 1$ and, on $\Omega_n$,
$$
	\sup_{S\in\cS_n} \sup_{\eta\in\cH_n} \left( h^T G_{n,\eta,S} - \half h^T V_{n,\eta,S}h \right) \leq -C \|h\|_2^2
$$
for every $h\in A_S$.
Let $\mu$ be the Lebesgue measure, then on $\Omega_n$, the numerator of \eqref{eq:normal-tail} is bounded by
\bean
	&&\int_{A_S} \exp\left[ -C\|h\|_2^2 \right] dh
	\leq
	\int_{\{h:\|h\|_2 > M_n \sqrt{s_n\log p}\}} \exp\left[ -C\|h\|_2^2 \right] dh
	\\
	&&\leq
	\sum_{k=1}^\infty \exp\left( -C k M_n^2 s_n\log p\right) \mu \left\{h: kM_n^2 s_n \log p < \|h\|_2^2 \leq (k+1) M_n^2 s_n \log p\right\}
	\\
	&&\leq
	\sum_{k=1}^\infty \exp\left( -C k M_n^2 s_n\log p\right) \mu \left\{h: \|h\|_2^2 \leq (k+1) M_n^2 s_n \log p\right\}
	\\
	&&=
	\sum_{k=1}^\infty \exp\left( -C k M_n^2 s_n\log p\right) \frac{\pi^{|S|/2}}{\Gamma(|S|/2 + 1)} \left\{(k+1) M_n^2 s_n \log p\right\}^{|S|}
	\\
	&&\leq
	(\sqrt{\pi} M_n^2 s_n\log p)^{s_n/2} \sum_{k=1}^\infty (k+1)^{s_n/2} \exp\left( -C k M_n^2 s_n\log p\right)
	\\
	&&\leq
	(\sqrt{\pi} M_n^2 s_n\log p)^{s_n/2} \sum_{k=1}^\infty \exp\left( - \half C k M_n^2 s_n\log p\right)
	\\
	&&\leq
	(\sqrt{\pi} M_n^2 s_n\log p)^{s_n/2} \left[ \exp\left(-\half CM_n^2 s_n \log p\right) + \int_1^\infty \exp\left(-\half Cx M_n^2 s_n \log p\right) dx \right]
	\\
	&&\leq
	(\sqrt{\pi} M_n^2 s_n\log p)^{s_n/2} \exp\left(-\hbox{$1 \over 3$} CM_n^2 s_n \log p\right)
\eean
for large enough $n$.
Note that the denominator of \eqref{eq:normal-tail} is equal to
$$	
	(2\pi)^{|S|/2} |V_{n,\eta,S}|^{-1/2} \exp\left( \hbox{$1\over{2v_{\eta_0}}$} \|\bH_S \dot L_{n,\eta} \|_2^2 \right),
$$
and
$$
	|V_{n,\eta,S}| \leq \left( \frac{{\rm tr}(V_{n,\eta,S})}{|S|}\right)^{|S|} \leq (L^2 v_{\eta})^{s_n/2},
$$
where ${\rm tr}(A)$ denotes the trace of a matrix $A$.
Therefore, the log of the left hand side of \eqref{eq:normal-tail} tends to $-\infty$ on $\Omega_n$, which completes the proof.
\qed

\begin{lemma} \label{lem:score-unif-rate}
For a sequence $\epsilon_n \rightarrow 0$, let $\cH_n = \{\eta\in\cH_0: d_H(\eta,\eta_0) \leq \epsilon_n\}$.
Then, for every $\zeta > 0$, 
$$
	\int \sup_{\eta\in\cH_n} (\score_\eta - \score_{\eta_0})^2 dP_{\eta_0} \leq
	K_\zeta (\epsilon_n)^{4/5 - \zeta}
$$
for large enough $n$, where $K_\zeta$ is a constant depending only on $\sigma_1, \sigma_2, M$ and $\zeta$.
\end{lemma}
\begin{proof}
For a function $f:\bbR\mapsto\bbR$, denote its first and second derivatives as $f^\prime$ and $f^{\prime\prime}$.
Note that $\score_\eta(y) = -\ell^\prime(y)$.
Note also that
$$
	\sup_{\eta\in\cH_0} |\score_\eta(y)| \lesssim |y| \quad {\rm and} \quad \eta_0(y) \lesssim e^{-a_1 y^2}
$$
for large enough $|y|$, where $a_1 > 0$ is a constant depending only on $\sigma_1,\sigma_2$ and $M$.
For a constant $C_1 > 0$, let $A = \{y: |y| \leq C_1\sqrt{\log 1/\epsilon_n}\}$.
Then,
\bean
	\int_{A^c} \sup_{\eta\in\cH_n} |\score_\eta(y) - \score_{\eta_0}(y)|^2 dP_{\eta_0}(y)
	\lesssim \int_{A^c} y^2 e^{-a_1 y^2} dy \leq \int_{A^c} y e^{-a_1 y^2/2} dy
	\lesssim \epsilon_n^{a_1 C_1^2/2}
\eean
for large enough $n$.
Thus, we can choose $C_1 > 0$, depending only on $\sigma_1,\sigma_2$ and $M$, such that 
$$
	\int_{A^c} \sup_{\eta\in\cH_n} |\score_\eta(y) - \score_{\eta_0}(y)|^2 dP_{\eta_0}(y)
	\leq \epsilon_n
$$
for large enough $n$.
Write $\ell^\prime_\eta(y) - \ell_{\eta_0}^\prime(y) = \{\ell^\prime_\eta(y) - \ell_{\eta_0}^\prime(y) - d_\eta(x,y)\} + d_\eta(x,y)$, where
$$
	d_\eta(x,y) = \frac{\ell_\eta(y+x) - \ell_\eta(y)}{x} - \frac{\ell_{\eta_0}(y+x) - \ell_{\eta_0}(y)}{x}.
$$
Note that by \eqref{eq:libschitz} and the Taylor expansion, $|\ell^\prime_\eta(y) - \ell_{\eta_0}^\prime(y) - d_\eta(x,y)| \lesssim |x|(1+|y|^2)$ provided that $|x|$ is small enough.
Also, $|x d_\eta(x,y)| \leq |\ell_\eta(y+x) - \ell_{\eta_0}(y+x)| + |\ell_\eta(y) - \ell_{\eta_0}(y)|$, and
\bean
	&& \int_A \sup_{\eta\in\cH_n} |\ell_\eta(y+x) - \ell_{\eta_0}(y+x)|^2 dP_{\eta_0}(y)
	\\
	&=& \int_A \sup_{\eta\in\cH_n} |\ell_\eta(y+x) - \ell_{\eta_0}(y+x)|^2 \frac{\eta_0(y)}{\eta_0(y+x)} \eta_0(y+x) dy
	\\
	&\lesssim& \int_A \sup_{\eta\in\cH_n} |\ell_\eta(y) - \ell_{\eta_0}(y)|^2 e^{y} dP_{\eta_0}(y)
\eean
for every small enough $x$, where the last inequality holds because
$$
	\frac{\eta_0(y)}{\eta_0(y+x)} \leq \sup_{|z|\leq M} \sup_{\sigma\in[\sigma_1,\sigma_2]} \frac{\exp\{-(y-z)^2/(2\sigma^2)\}}{\exp\{-(y+x-z)^2/(2\sigma^2)\}} \lesssim e^{|y|} \leq e^y + e^{-y}.
$$
Therefore,
\be\begin{split} \label{eq:sup-score-bound}
	&\int \sup_{\eta\in\cH_n} |\score_\eta(y) - \score_{\eta_0}(y)|^2 dP_{\eta_0}(y)
	\\
	&\lesssim \epsilon_n + |x|^2 + \frac{\int_A \sup_{\eta\in\cH_n} |\ell_\eta(y) - \ell_{\eta_0}(y)|^2 e^y dP_{\eta_0}(y)}{|x|^2}
\end{split}\ee
for every small enough $x$.

Let $f_\eta(y) = \{\ell_\eta(y) - \ell_{\eta_0}(y)\}^2 \eta_0(y)$ and $\delta_n = \epsilon_n \log(1/\epsilon_n)$.
Since
\be\begin{split}\label{eq:df-bound}
	f^\prime_\eta(y) = & 2\{\ell_\eta(y) - \ell_{\eta_0}(y)\} \{\ell^\prime_\eta(y) - \ell^\prime_{\eta_0}(y)\} \eta_0(y) 
	\\
	& + \{\ell_\eta(y) - \ell_{\eta_0}(y)\}^2 \eta_0^\prime(y),
\end{split}\ee
we have by the triangle inequality
\bean
	|f^\prime_\eta(y)|
	&\lesssim& \sqrt{f_\eta(y)} \left( |\ell_\eta^\prime(y) - \ell_{\eta_0}^\prime(y)| \sqrt{\eta_0(y)} + |\ell_\eta(y) - \ell_{\eta_0}(y)| \frac{|\eta_0^\prime(y)|}{\sqrt{\eta_0(y)}} \right)
	\\
	&=& \sqrt{f_\eta(y)} \sqrt{\eta_0(y)} \Big( |\ell_\eta^\prime(y) - \ell_{\eta_0}^\prime(y)| + |\ell_\eta(y) - \ell_{\eta_0}(y)| \ell_{\eta_0}^\prime(y) \Big)
	\\
	&\lesssim& (1+|y|^3) \sqrt{\eta_0(y)} \sqrt{f_\eta(y)}.
\eean
Also, it is easy to show that $|f^{\prime\prime}_\eta(y)| \lesssim (1+|y|^4) \eta_0(y)$ and $\eta_0(y+x) / \eta_0(y) \lesssim e^y$ for small enough $x$.
Thus, by the Taylor expansion, we have that
\bean
	|f_\eta(y+x) - f_\eta(y)| &\lesssim& |x|(1+|y|^3) \sqrt{\eta_0(y)} \sqrt{f_\eta(y)} + x^2 (1+|y|^4) e^y \eta_0(y),
\eean
and therefore, there exists a constant $C_2 > 1/4$, depending only on $\sigma_1, \sigma_2$ and $M$, such that
$$
	|f_\eta(y+x) - f_\eta(y)| \leq C_2 \left( |x| \sqrt{f_\eta(y)} + x^2\right)
$$
for every $\eta \in \cH_0$, $y\in\bbR$ and small enough $x$.
Assume that for some constant $K$, $f_\eta(y_0) > K\delta_n^{4/3}$ for some $\eta\in\cH_n$ and $y_0\in\bbR$.
Then $f_\eta(y_0+x) > f_\eta(y_0)/2$ for every $x$ with $|x| \leq \sqrt{f_\eta(y_0)} / 4C_2$, so it holds that
$$
	\int f_\eta(y) dy \geq \int_{\{y: |y-y_0| < \sqrt{f_{\eta}(y_0)}/(4C_2)\}} f_\eta(y) dy
	\geq \frac{K^{3/2}}{4C_2} \delta_n^2.
$$
If $K$ is large enough, this makes a contradiction because $\int f_\eta(y) dy \lesssim \delta_n^2$ for every $\eta\in\cH_n$ by Theorem 5 of \cite{wong1995probability}.
Therefore, it holds that $f_\eta(y) \lesssim \delta_n^{4/3}$ for every $\eta\in\cH_n$ and $y\in\bbR$.

Next, we claim that if there exists a $\gamma \in (1,2)$ such that $f_\eta(y) \lesssim (\delta_n)^\gamma$ for every $y\in B$, then $f_\eta(y) \lesssim (\delta_n)^{1+(3\gamma)/8}$ for every $y\in B$, where $B = \{y: |y| \leq 2C_1 \sqrt{\log(1/\delta_n)}\}$.
For every $y\in B$ and small enough $x$ with $y+x\in B$, it holds by \eqref{eq:libschitz} that
\bean
	|\ell_\eta^\prime(y) - \ell_{\eta_0}^\prime(y)| 
	&\leq& \left|\ell_\eta^\prime(y) - \ell_{\eta_0}^\prime(y) - \left(\frac{\ell_\eta(y+x) - \ell_\eta(y)}{x} - \frac{\ell_{\eta_0}(y+x) - \ell_{\eta_0}(y)}{x} \right) \right|
	\\
	&& \quad + \left|\frac{\ell_\eta(y+x) - \ell_\eta(y)}{x} - \frac{\ell_{\eta_0}(y+x) - \ell_{\eta_0}(y)}{x} \right|
	\\
	&\lesssim& |x| (1+|y|^2) + \frac{|\ell_\eta(y+x) - \ell_{\eta_0}(y+x)| + |\ell_\eta(y) - \ell_{\eta_0}(y)|}{|x|}.
\eean
Since
$$
	|\ell_\eta(y+x) - \ell_{\eta_0}(y+x)| \sqrt{\eta_0(y)} = \sqrt{f_\eta(y+x)} \sqrt{\frac{\eta_0(y)}{\eta_0(y+x)}} \lesssim e^{|xy|/\sigma_1^2} (\delta_n)^{\gamma/2},
$$
we have that
\be\label{eq:score-sqrt-prod}
	|\ell_\eta^\prime(y) - \ell_{\eta_0}^\prime(y)| \sqrt{\eta_0(y)} \lesssim |x| (1+|y|^2) \sqrt{\eta_0(y)} + \frac{e^{|xy|/\sigma_1^2} (\delta_n)^{\gamma/2}}{|x|}
	\lesssim |x| + \frac{e^{|xy|/\sigma_1^2} (\delta_n)^{\gamma/2}}{|x|}.
\ee
By taking $|x| = (\delta_n)^{\gamma/4}$, the right hand side of \eqref{eq:score-sqrt-prod} is bounded by a constant multiple of $(\delta_n)^{\gamma/4}$ for every $y\in B$.
Therefore, by \eqref{eq:df-bound}, there exists a constant $C_3 > 0$, depending only on $\sigma_1, \sigma_2$ and $M$, such that
$$
	|f^\prime_\eta(y)| \lesssim \sqrt{f_\eta(y)} (\delta_n)^{\gamma/4} + (1+|y|) f_\eta(y)
	\leq C_3(\delta_n)^{(3\gamma)/4}
$$
for $y \in B$.
As before, for some large constant $K$, assume that $f_\eta(y_0) > K\delta_n^{1+(3\gamma)/8}$ for some $\eta\in\cH_n$ and $y_0\in B$.
Then, $|f_\eta(y+x) - f_\eta(y)| \leq K\delta_n^{1+(3\gamma)/8}/2$ provided that $|x| \leq K(\delta_n)^{1-(3\gamma)/8} / (2C_3)$.
Since $\int f_\eta(y) dy \lesssim \delta_n^2$ by Theorem 5 of \cite{wong1995probability},
large $K$ makes a contradiction, so the proof of claim is complete.

Assume that $\zeta > 0$ is given.
Since the real sequence $(\gamma_k)_{k=1}^\infty$ defined as $\gamma_1 = 4/3$ and $\gamma_{k+1} = 1 + (3 \gamma_k)/8$ converges to $8/5$, by applying the claim repeatedly, we can find a constant $C_\zeta > 0$ depending only on $\sigma_1, \sigma_2, M$ and $\zeta$ such that $|f_\eta(y)| \leq C_\zeta (\delta_n)^{8/5-\zeta}$.
Also, for large enough $n$, $\delta_n \leq \sqrt{\epsilon_n}$ implies that $\sqrt{\log (1/\epsilon_n)} \leq 2\sqrt{\log (1/\delta_n)}$.
Therefore,
\bean
	&&\int_A \sup_{\eta\in\cH_n} |\ell_\eta(y) - \ell_{\eta_0}(y)|^2 e^y dP_{\eta_0}(y)
	\leq \int_B \sup_{\eta\in\cH_n} |\ell_\eta(y) - \ell_{\eta_0}(y)|^2 e^y dP_{\eta_0}(y)
	\\
	&&\leq C_\zeta \delta_n^{8/5 - \zeta} \int_B e^y dy
	\leq C_\zeta \delta_n^{8/5 - \zeta} \int_{-\infty}^{2 C_1 \sqrt{\log \delta_n^{-1}}} e^y dy
	\\
	&&= C_\zeta \delta_n^{8/5 - \zeta} e^{2 C_1 \sqrt{\log \delta_n^{-1}}}
	\leq  C_\zeta (\delta_n)^{8/5 - 2\zeta}
\eean
for large enough $n$.
Note that the right hand side of \eqref{eq:sup-score-bound} is minimized when
$$
	|x|^4 = \int_A \sup_{\eta\in\cH_n} |\ell_\eta(y) - \ell_{\eta_0}(y)|^2 e^y dP_{\eta_0}(y).
$$
In this case,  \eqref{eq:sup-score-bound} is bounded by a constant multiple of $(\delta_n)^{4/5 - \zeta}$.
Since $(\delta_n)^{4/5 - \zeta} \leq \epsilon_n^{4/5 - 2\zeta}$, the proof is complete by \eqref{eq:sup-score-bound}.
\end{proof}

\medskip
{\it Proof of Theorem \ref{thm:LAN}.}
By Theorem \ref{thm:mLAN}, $s_n^5 (\log p)^3 = o(n)$ implies that
$$
	\bbE \Big(\sup_{\theta\in M_n\Theta_n} \sup_{\eta\in\cH_n} |\widetilde r_n(\theta,\eta)|\Big) = o(1)
$$
for some $M_n \rightarrow \infty$, where
$$
	\widetilde r_n(\theta,\eta) = L_n(\theta,\eta) - L_n(\theta_0,\eta)
	- \sqrt{n}(\theta-\theta_0)^T \bbG_n \score_{\theta_0,\eta}
	+ \frac{n}{2} (\theta-\theta_0)^T V_{n,\eta}(\theta-\theta_0).
$$
Therefore, it suffices to prove that
\be\label{eq:lan-tech1}
	\sup_{\theta\in M_n\Theta_n} \sup_{\eta\in\cH_n} \frac{n}{2} \left| (\theta-\theta_0)^T (V_{n,\eta} - V_{n,\eta_0}) (\theta-\theta_0)\right| = o(1)
\ee
and
\be\label{eq:lan-tech2}
	\bbE\left( \sup_{\theta\in M_n\Theta_n} \sup_{\eta\in\cH_n} \left|\sqrt{n}(\theta-\theta_0)^T \bbG_n (\score_{\theta_0, \eta} - \score_{\theta_0, \eta_0})\right| \right) = o(1).
\ee
We may assume that $M_n$ is sufficiently slowly increasing as described below.
Note that
\bean
	|v_\eta - v_{\eta_0}| \| \bX(\theta-\theta_0)\|_2^2 
	\leq K_{\rm theta}^2 |v_\eta - v_{\eta_0}| M_n^2 s_n \log p
	\\
	\lesssim K_{\rm theta}^2 \left\{\int (\score_\eta - \score_{\eta_0})^2 dP_{\eta_0}\right\}^{1/2} M_n^2 s_n \log p,
\eean
where the second inequality holds by Cauchy-Schwarz.
Since for every $\zeta > 0$ there exists a constant $C_\zeta > 0$ such that
$$
	\left\{\int (\score_\eta - \score_{\eta_0})^2 dP_{\eta_0}\right\}^{1/2} \lesssim C_\zeta \left(\frac{s_n \log p}{n}\right)^{1/5-\zeta},
$$
by Lemma \ref{lem:score-unif-rate} and $(s_n\log p)^6  = O(n^{1-\xi})$ for some $\xi > 0$, \eqref{eq:lan-tech1} holds for sufficiently slowly growing $M_n$.

For \eqref{eq:lan-tech2}, we may assume that $\theta_0=0$ without loss of generality.
Note that
\bean
	\left|\sqrt{n}(\theta-\theta_0)^T \bbG_n (\score_{\theta_0, \eta} - \score_{\theta_0, \eta_0})\right|
	\leq \sqrt{n}\|\theta-\theta_0\|_1 \|\bbG_n (\score_{\theta_0, \eta} - \score_{\theta_0, \eta_0})\|_\infty
	\\
	\leq K_{\rm theta} M_n s_n\sqrt{\log p} \sup_{\eta\in\cH_n} \Big\| \bbG_n (\score_{\theta_0, \eta} - \score_{\theta_0, \eta_0}) \Big\|_\infty
\eean
for every $\theta \in M_n \Theta_n$ and $\eta\in\cH_n$.
Let $\cF_n = \cup_{j=1}^p \cF_{n,j}$, where
$$
	\cF_{n,j} = \left\{M_n s_n \sqrt{\log p}\; e_j^T(\score_{\theta_0, \eta} - \score_{\theta_0, \eta_0}): \eta\in\cH_n \right\}
$$
and $e_j$ is the $j$th unit vector in $\bbR^p$.
Note that 
$$
	F_n(x,y) = L M_n s_n \sqrt{\log p} \sup_{\eta\in\cH_n} |\score_{\eta}(y) - \score_{\eta_0}(y)|
$$
is an envelope function of $\cF_n$.
Also, it is easy to see that
$$
	N_{[]}^n (\epsilon, \cF_{n,j}) \leq N_{[]}\left(\frac{\epsilon}{L M_n s_n \sqrt{\log p}}, \cG_n, L_2(P_{\eta_0}) \right),
$$
where $\cG_n = \{ \score_\eta: \eta\in \cH_n \}$.
It follows that
\bean
	\log N_{[]}^n(\epsilon, \cF_n) \leq \log p + \log N_{[]}\left(\frac{\epsilon}{L M_n s_n \sqrt{\log p}}, \cG_n, L_2(P_{\eta_0}) \right).
\eean
Similarly to the proof of Lemma \ref{lem:score-unif-rate}, it can be shown that $L_2(P_{\eta_0})$-norm of $\score_{\eta_1} - \score_{\eta_2}$ is bounded by $\{d_H(\eta_1, \eta_2)\}^\gamma$ for some constant $\gamma > 0$, so
\bean
	\log N_{[]}(\epsilon, \cG_n, L_2(P_{\eta_0})) \leq \log N_{[]}(\epsilon^{1/\gamma}, \cH_n, d_H) \lesssim  \Big( \frac{1}{\gamma} \log \frac{1}{\epsilon}\Big)^3
\eean
for every $\epsilon > 0$, where the last inequality holds by Theorem 3.3 of \cite{ghosal2001entropies}.
This implies that
\bean
	\log N_{[]}^n(\epsilon, \cF_n) \lesssim \gamma^{-3} \left\{ \Big(\log \frac{1}{\epsilon}\Big)^3 + (\log s_n)^3 + \log p\right\}.
\eean
Since for every $\zeta >0$ there exists a constant $C^\prime_\zeta > 0$ such that
$$
	\|F_n\|_n \lesssim M_n s_n \sqrt{\log p} \left\{ \int \sup_{\eta\in\cH_n} (\score_{\eta} - \score_{\eta_0})^2 dP_{\eta_0}\right\}^{1/2} 
	\leq C_\zeta^\prime M_n s_n \sqrt{\log p} \left( \frac{s_n \log p}{n}\right)^{1/5-\zeta}
$$
by Lemma \ref{lem:score-unif-rate} and
$$
	\int_0^a \left(\log \frac{1}{\epsilon}\right)^{3/2} d\epsilon 
	\leq \int_0^a \left(\log \frac{1}{\epsilon}\right)^2 d\epsilon
	= \int_{-\log a}^\infty x^2 e^{-x} dx
	\leq \int_{-\log a}^\infty e^{-x/2} dx \leq 2\sqrt{a}
$$
for small enough $a > 0$, we have
\be\begin{split} \label{eq:sup_bd}
	\bbE \Big( \sup_{f\in\cF_n} |\bbG_nf|\Big) 
	&\lesssim
	\gamma^{-3/2} C_\zeta^\prime M_n s_n \sqrt{\log p} \left( \frac{s_n \log p}{n}\right)^{1/5-\zeta}
	\left\{ \sqrt{\log p} + (\log s_n)^{3/2} \right\}
	\\
	& + \gamma^{-3/2} \left\{ C_\zeta^\prime M_n s_n \sqrt{\log p} \left( \frac{s_n \log p}{n}\right)^{1/5-\zeta} \right\}^{1/2}
\end{split}\ee
for every $\zeta > 0$ by Corollary \ref{cor:maximal_bracket}.
Note that
$$
	s_n \log p \left( \frac{s_n \log p}{n}\right)^{1/5-\zeta} = 
	n^{-\xi/5} \left\{ \frac{(s_n \log p)^6}{n^{1-\xi}} \right\}^{1/5}
	\left( \frac{s_n \log p}{n}\right)^{-\zeta}.
$$
Since $(s_n\log p)^6  = O(n^{1-\xi})$ and $\log s_n \leq \log n$, \eqref{eq:sup_bd} is of order $o(1)$ provided that $\zeta$ is small enough and $M_n$ is sufficiently slowly growing.
\qed

\begin{lemma} \label{lem:d_v-partition}
Let $(\Omega, \cF)$ be a measurable space and $(\Omega_i)_{i\in I}$ be a measurable partition of $\Omega$ for some discrete index set $I$.
Let $\bw = (w_i)_{i\in I}$ and $\widetilde\bw = (\widetilde w_i)_{i\in I}$ be probability measures on $I$, and for each $i\in I$, $P_i$ and $\widetilde P_i$ be probability measures on $\Omega$ such that $P_i(\Omega_i) = \widetilde P_i (\Omega_i) = 1$.
Then, it holds that
$$
	d_V(Q, \widetilde Q) \leq 2d_V(\bw, \widetilde \bw) + \sum_{i\in I} w_i d_V(P_i, \widetilde P_i)
$$
where $Q = \sum_{i\in I} w_i P_i$ and $\widetilde Q = \sum_{i\in I} \widetilde w_i \widetilde P_i$.
\end{lemma}
\begin{proof}
Let $\overline Q = \sum_{i\in I} w_i\widetilde P_i$, then $d_V(Q, \widetilde Q) \leq d_V(Q, \overline Q) + d_V(\overline Q, \widetilde Q)$ by the triangle inequality.
Then, for any $A \in \cF$,
\bean
	2|Q(A) - \overline Q(A)| = 2\left| \sum_{i\in I} w_i \left\{ P_i(A\cap \Omega_i) - \widetilde P_i(A \cap \Omega_i) \right\}\right| \leq \sum_{i\in I} w_i d_V(P_i, \widetilde P_i).
\eean
Also, for any $A \in \cF$
\bean
	|\overline Q(A) - \widetilde Q(A)| = \left| \sum_{i\in I} (w_i - \widetilde w_i) \widetilde P_i(A \cap \Omega_i) \right|
	\leq \sum_{i\in I} |w_i - \widetilde w_i| = d_V(\bw, \widetilde \bw),
\eean
so the proof is complete.
\end{proof}

\medskip
{\it Proof of Theorem \ref{thm:BvM}.}
Define $r_n(\theta,\eta)$ as in Theorem \ref{thm:LAN}, then there is $M_n \rightarrow \infty$ such that 
$$
	\sup_{\theta\in M_n\Theta_n}\sup_{\eta\in\cH_n} |r_n(\theta,\eta)| = o_{P_0}(1).
$$
Note that $(M_n)$ can be chosen to be sufficiently slowly increasing, so that 
$$
	\sup_{S\in\cS_n}\sup_{\theta_S\in\Theta_S} \lambda \|\theta_S-\theta_{0,S}\|_1 = o(1),
$$
where
$\Theta_S = \{\theta_S\in\bbR^{|S|}: \widetilde\theta_S\in M_n\Theta_n\}$.
Therefore, we have 
\be\label{eq:small-lambda-bvm}
	\sup_{S\in\cS_n} \sup_{\theta_S\in\Theta_S} |\log \{g_S(\theta_S) / g_S(\theta_{0,S})\}| = o(1).
\ee

Let $\widetilde \Pi_\Theta$ and $\widetilde\Pi_\cH$ be priors restricted and renormalized on $M_n\Theta_n$ and $\cH_n$, respectively.
Let $\widetilde\Pi = \widetilde\Pi_\Theta \times \widetilde \Pi_\cH$ and $\widetilde \Pi(\cdot| \bD_n)$ be the corresponding posterior distribution.
Then, it is easy to see that $d_V(\widetilde \Pi(\cdot|\bD_n), \Pi(\cdot|\bD_n)) = o_{P_0}(1)$.
Similarly, let $\widetilde \Pi^\infty(\cdot|\bD_n)$ be the restricted and renormalized version of $\Pi^\infty(\cdot|\bD_n)$ onto $M_n \Theta_n$.
It can be written as
\bean
	d\widetilde\Pi(\theta|\bD_n) &=& \sum_{S\in\cS_n} \widetilde w_S d\widetilde Q_S(\theta_S) d\delta_0(\theta_{S^c}),
	\\
	d\widetilde\Pi^\infty(\theta|\bD_n) &=& \sum_{S\in\cS_n} n^{-|S|/2} \widetilde w_S^{\infty} d\widetilde \cN_{n,S}(h_S) d\delta_0(\theta_{S^c}),
\eean
where $\widetilde Q_S$ and $\widetilde \cN_{n,S}$ are restricted and renormalized versions of $Q_S$ and $\cN_{n,S}$ onto $\Theta_S$ and $H_S = \sqrt{n}(\Theta_S - \theta_0)$, respectively, and 
\bean
	\widetilde w_S = \frac{Q_S(\Theta_S)}{\sum_{S^\prime \in \cS_n} w_{S^{\prime}} Q_{S^\prime} (\Theta_{S^\prime})} w_S
	\quad \textrm{and} \quad
	\widetilde w_S^\infty = \frac{\cN_{n,S}(H_S)}{\sum_{S^\prime \in \cS_n} w_{S^\prime} \cN_{n,S^\prime} (H_{S^\prime})} w_S.
\eean
Since $\sum_{S\in\cS} w_S \rightarrow 1$ in $P_0^{(n)}$-probability, we have, by Lemma \ref{lem:normal-concentration}, that
\be\label{eq:w-ratio}
	\sup_{S\in\cS_n} \left| 1 - \frac{w_S}{\widetilde w_S^\infty} \right| = o_{P_0}(1)
\ee
and
\bean
	\sup_{S\in\cS_n} d_V(\cN_{n,S}, \widetilde \cN_{n,S}) = o_{P_0}(1).
\eean
It follows by Lemma \ref{lem:d_v-partition} that $d_V(\widetilde \Pi^\infty(\cdot|\bD_n), \Pi^\infty(\cdot|\bD_n)) = o_{P_0}(1)$.
Therefore, it suffices to prove \eqref{eq:bvm2} with $\Pi(\cdot|\bD_n)$ and $\Pi^\infty(\cdot|\bD_n)$ replaced by $\widetilde\Pi(\cdot|\bD_n)$ and $\widetilde\Pi^\infty(\cdot|\bD_n)$, respectively.

Let $\widetilde\bw=(\widetilde w_S)_{S\in\cS_n}$ and $\widetilde\bw^\infty=(\widetilde w^\infty_S)_{S\in\cS_n}$.
Note that $\sum_{S \in \cS_n} w_{S} Q_{S} (\Theta_{S}) \rightarrow 1$ in $P_0^{(n)}$-probability.
Thus, by Lemma \ref{lem:normal-concentration} and \eqref{eq:w-ratio}, we have
\bean
	d_V(\widetilde \bw, \widetilde \bw^\infty) 
	&=& \sum_{S\in\cS_n} |\widetilde w_S - \widetilde w_S^\infty| 
	= \sum_{S_\in\cS_n} \bigg|1 - \frac{\widetilde w_S}{\widetilde w_S^\infty} \bigg| \widetilde w_S^\infty 
	\\
	&=& \sum_{S_\in\cS_n} \Big\{1 - Q_S(\Theta_S) \Big\} \widetilde w_S^\infty  + o_{P_0}(1)
	= o_{P_0}(1).
\eean
Since
\bean
	\widetilde \Pi(\theta \in B | \bD_n, \eta, S_\theta=S)
	= \frac{\int_{B\cap\Theta_S} \exp\{L_n(\widetilde \theta_S, \eta) - L_n(\theta_0,\eta)\} g_S(\theta_S)/g_S(\theta_{0,S}) d\theta_S}{\int_{\Theta_S} \exp\{ L_n(\widetilde \theta_S, \eta) - L_n(\theta_0,\eta)\} g_S(\theta_S)/g_S(\theta_{0,S}) d\theta_S},
\eean
we have by Theorem \ref{thm:LAN}, Lemma \ref{lem:normal-concentration} and \eqref{eq:small-lambda-bvm} that
\be\label{eq:cond-bvm}
	\sup_{\eta\in\cH_n} \sup_{S\in\cS_n} \sup_B \left| \widetilde\Pi(h_S\in B| \bD_n, \eta, S_\theta=S) - \widetilde \cN_{n,S}(B)\right| = o_{P_0}(1),
\ee
where the third supremum is taken over all measurable $B\subset\bbR^{|S|}$.
Since
$$
	\widetilde\Pi(h_S\in B| \bD_n, S_\theta=S) = \int_{\cH_n} \widetilde\Pi(h_S\in B| \bD_n, \eta, S_\theta=S) d\widetilde\Pi(\eta|\bD_n, S_\theta=S),
$$
we have 
$$
	\sup_{S\in\cS_n} \sup_B \big|\widetilde Q_S(h_S\in B) - \widetilde \cN_{n,S}(B)) \big| = o_{P_0}(1).
$$
Therefore, \eqref{eq:bvm2} holds by Lemma \ref{lem:d_v-partition}.
\qed

\medskip
{\it Proof of Theorem \ref{thm:selection}.}
Let $\cS_n^\prime = \{S\in\cS_n: S \supsetneq S_0\}$ and $(M_n)$ be a diverging sequence satisfying the assertion of Theorem \ref{thm:LAN}.
Note that $(M_n)$ can be chosen to be sufficiently slowly increasing, so that $\sup_{S\in\cS_n} \sup_{\theta_S\in \Theta_S} \lambda \|\theta_S-\theta_{0,S}\|_1 = o(1)$, where $\Theta_S = \{\theta_S\in\bbR^{|S|}: \widetilde\theta_S\in M_n\Theta_n\}$.
Then it holds that
\be\label{eq:small-lambda}
	\sup_{S\in\cS_n} \sup_{\theta_S\in\Theta_S} |\log \{g_S(\theta_S) / g_S(\theta_{0,S})\}| = o(1).
\ee
As in the proof of Theorem \ref{thm:BvM}, let $\widetilde \Pi(\cdot| \bD_n)$ be the posterior distribution based on the restricted and renormalized priors $\widetilde \Pi = \widetilde \Pi_\Theta \times \widetilde\Pi_\cH$ on $M_n\Theta_n\times\cH_n$.
Then, we have $d_V(\Pi(\cdot|\bD_n), \widetilde\Pi(\cdot|\bD_n)) \rightarrow 0$ in $P_0^{(n)}$-probability.
Thus, it suffices to prove that $\E \widetilde\Pi(S_\theta \in\cS_n^\prime | \bD_n) \rightarrow 0$.

By Theorem \ref{thm:LAN} and Lemma \ref{lem:normal-concentration}, there exist events $(\Omega^\prime_n)$ and a sequence $\epsilon_n \rightarrow 0$ such that $P_0^{(n)}(\Omega^\prime_n) \rightarrow 1$ and on $\Omega^\prime_n$,
\be\begin{split}\nonumber
	& \exp\left\{ \sqrt{n}(\theta-\theta_0)^T \bbG_n \score_{\theta_0,\eta_0}
	- \frac{n}{2} (\theta-\theta_0)^T V_{n,\eta_0}(\theta-\theta_0) - \epsilon_n\right\}
	\\
	& \leq \exp\left\{ L_n(\theta,\eta) - L_n(\theta_0, \eta)\right\} 
	\\
	&\leq \exp\left\{ \sqrt{n}(\theta-\theta_0)^T \bbG_n \score_{\theta_0,\eta_0}
	- \frac{n}{2} (\theta-\theta_0)^T V_{n,\eta_0}(\theta-\theta_0) + \epsilon_n\right\} 
\end{split}\ee
for every $\theta\in M_n \Theta_n$ and $\eta\in\cH_n$, and 
\be \label{eq:normal-tail2}
	\inf_{S\in\cS_n} \frac{\int_{\Theta_S} \exp\left( h^T G_{n,S} - \half h^T V_{n,S}h\right) dh}
	{\int_{\bbR^{|S|}} \exp\left( h^T G_{n,S} - \half h^T V_{n,S}h\right) dh} \geq e^{-\epsilon_n}.
\ee
Since
\bean
	\widetilde\Pi(S_\theta = S | \bD_n, \eta) \propto \frac{\pi_p(|S|)}{\binom{p}{|S|}} \int_{\Theta_S} \exp\left\{ L_n(\widetilde\theta_S,\eta) - L_n(\theta_0,\eta) \right\} g_S(\theta_S) d\theta_S
\eean
for every $S \in \cS_n^\prime$ and $\eta\in\cH_n$, we have on $\Omega^\prime_n$ that $\widetilde \Pi(S_\theta \in \cS_n^\prime | \bD_n, \eta) \leq e^{2\epsilon_n} B_n / A_n$, where
\bean
	A_n &=& \frac{\pi_p(s_0)}{\binom{p}{s_0}} \int_{\Theta_{S_0}} \exp\left\{ E_{n,S_0}(\theta) \right\} g_{S_0}(\theta_{S_0}) d\theta_{S_0}
	\\
	B_n &=& \sum_{S\in\cS_n^\prime}\frac{\pi_p(|S|)}{\binom{p}{|S|}} \int_{\Theta_{S}} \exp\left\{ E_{n,S}(\theta) \right\} g_{S}(\theta_{S}) d\theta_{S}
\eean
and
$$
	E_{n,S}(\theta) =  \sqrt{n}(\theta_{S}-\theta_{0,S})^T G_{n,S} - \frac{n}{2} (\theta_{S}-\theta_{0,S})^T V_{n,S}(\theta_{S}-\theta_{0,S}).
$$
Note that both $A_n$ and $B_n$ do not depend on $\eta$.
Thus, by \eqref{eq:small-lambda} and \eqref{eq:normal-tail2}, if $\epsilon_n$ is sufficiently slowly decreasing, then on $\Omega^\prime_n$, we have
\be\label{eq:w-hat-ratio}
	e^{-3\epsilon_n} \widetilde \Pi(S_\theta \in \cS_n^\prime | \bD_n) \leq \sum_{S\in\cS_n^\prime} \frac{\hat w_S}{\hat w_{S_0}},
\ee
where $\hat w_S$ is defined as \eqref{eq:w-hat-def}.
Therefore, the proof is complete if the right hand side of \eqref{eq:w-hat-ratio} is of order $o_{P_0}(1)$.

The right hand side of \eqref{eq:w-hat-ratio} is bounded by
\bean
	\sum_{s=s_0 + 1}^{s_n/2} \frac{\pi_p(s)}{\pi_p(s_0)} \binom{s}{s_0} \left(\frac{\lambda \sqrt{\pi}}{\sqrt{2 v_{\eta_0}}}\right)^{s-s_0} 
	\max_{|S|=s} \left[\frac{|\bX_{S_0}^T \bX_{S_0}|^{1/2}}{|\bX_S^T \bX_S|^{1/2}}
	\exp\left\{\hbox{$1 \over {2v_{\eta_0}}$} \|(\bH_S - \bH_{S_0}) \dot L_{n,\eta_0}\|_2^2 \right\} \right],
\eean
and
$$
	\frac{\pi_p(s)}{\pi_p(s_0)} \leq A_2^{s-s_0} p^{-A_4(s-s_0)}
$$
by \eqref{eq:pi_p_condition}.
It is shown in \cite{castillo2015bayesian} (see (6.11)) that
$$
	\frac{|\bX_{S_0}^T \bX_{S_0}|}{|\bX_S^T \bX_S|} \leq \{n\psi^2(s_n)\}^{-(|S|-s_0)}
$$
for every $S \in \cS_n^\prime$.
Also, we shall show below that
\be\label{eq:projection}
	P_0^{(n)} \left(\hbox{$1 \over {2v_{\eta_0}}$} \|(\bH_S - \bH_{S_0}) \dot L_{n,\eta_0} \|_2^2 > K_{\rm sel} (s-s_0) \log p,\; \textrm{for some $S \in \cS_n^\prime$}\right) \rightarrow 0
\ee
for some constant $K_{\rm sel}$, depending only on $\eta_0$.
Therefore, for some constant $C$, the right hand side of \eqref{eq:w-hat-ratio} is bounded by
\bean
	\sum_{s=1}^\infty e^{ -(s-s_0)\left\{ A_4 \log p + \log n - \log s - \log \lambda - K_{\rm sel} \log p + C \right\} }
\eean
with probability tending to 1, which converges to 0 provided that $A_4 > K_{\rm sel}$.

It only remains to prove \eqref{eq:projection}.
Note that the number of models $S$ containing $S_0$ with $|S|=s$ is equal to $N_s = \binom{p-s_0}{s-s_0}$.
By the Markov inequality, for any $r, u > 0$, 
\be\begin{split}\label{eq:proj-tech1}
	& P_0^{(n)} \left( \max_{|S| = s} \|(\bH_S - \bH_{S_0}) \dot L_{n,\eta_0}\|_2^2 > r\log N_s \right)
	\\
	&\leq e^{-ur\log N_s} \E \left(\max_{|S|=s} e^{u \|(\bH_S - \bH_{S_0}) \dot L_{n,\eta_0} \|_2^2}\right)
	\\
	&\leq N_s^{-(ur-1)} \max_{|S|=s} \E e^{u \|(\bH_S - \bH_{S_0}) \dot L_{n,\eta_0} \|_2^2}
\end{split}\ee
and
\bean
	\E e^{u \|(\bH_S - \bH_{S_0}) \dot L_{n,\eta_0} \|_2^2}
	\leq e^{u \E\|(\bH_S - \bH_{S_0}) \dot L_{n,\eta_0} \|_2^2}
	\E e^{u |\|(\bH_S - \bH_{S_0}) \dot L_{n,\eta_0} \|_2^2 - \E\|(\bH_S - \bH_{S_0}) \dot L_{n,\eta_0} \|_2^2|}.
\eean
For $S \in\cS_n^\prime$, there exists an orthonormal set $\{\bfe_{S, j}: j \leq |S|-s_0\}$ in $\bbR^n$ such that 
$$
	\|(\bH_S - \bH_{S_0}) \dot L_{n,\eta_0} \|_2^2 = \sum_{j=1}^{|S|-s_0} \left(\sum_{i=1}^n e_{S,ji} \score_{\eta_0}(\epsilon_i)\right)^2,
$$
where $\bfe_{S,j} = (e_{S,ji})$ and $\epsilon_i = Y_i - x_i^T \theta_0$.
Thus, 
$$
	\E \|(\bH_S - \bH_{S_0}) \dot L_{n,\eta_0} \|_2^2 = v_{\eta_0} (|S|-s_0).
$$
Since $\bH_S - \bH_{S_0}$ is an orthogonal projection matrix, it holds that $\|\bH_S - \bH_{S_0}\| \leq 1$ and $\|\bH_S - \bH_{S_0}\|_F = \sqrt{|S|-s_0}$, where $\|\cdot\|$ and $\|\cdot\|_F$ denote the $\ell_2$-operator norm and Frobenius norm.
Since $\score_{\eta_0}(\epsilon_i)$ is a sub-Gaussian random variable by \eqref{eq:libschitz} and $\E \score_{\eta_0}(\epsilon_i) = 0$, there exists a universal constant $c>0$ and a constant $K$ depending only on $\eta_0$, such that
\bean
	&&\E e^{u |\|(\bH_S - \bH_{S_0}) \dot L_{n,\eta_0} \|_2^2 - \E\|(\bH_S - \bH_{S_0}) \dot L_{n,\eta_0} \|_2^2|}
	\\
	&&= \int_1^\infty P_0^{(n)} \left( \big|\|(\bH_S - \bH_{S_0}) \dot L_{n,\eta_0} \|_2^2 - \E\|(\bH_S - \bH_{S_0}) \dot L_{n,\eta_0} \|_2^2\big| > \hbox{$1 \over u$} \log t \right) dt
	\\
	&&\leq \int_1^\infty 2\exp\left[ -c \min \left\{ \frac{(\log t / u)^2}{K^4 (|S|-s_0)}, \frac{\log t / u}{K^2} \right\} \right] dt
	\\
	&&\leq 2 e^{uK^2(|S|-s_0)} + 2\int_{e^{uK^2(|S|-s_0)}}^\infty e^{-\frac{c\log t}{uK^2}} dt
	\\
	&& \leq 2 e^{uK^2(|S|-s_0)} + 2\int_{e^{uK^2(|S|-s_0)}}^\infty t^{-\frac{c}{uK^2}} dt
\eean
where the first inequality holds by Lemma \ref{lem:hw-inequality}.
If we take $u= c/(2K^2)$ and $r = 2/u$, then the  last integral is bounded by $c/2$, so \eqref{eq:proj-tech1} is bounded by
$$
	2 N_s^{-1} e^{uv_{\eta_0}(|S|-s_0)} \left( e^{c(|S|-s_0)/2} + C \right)
$$
for some constant $C$.
With $K_{\rm sel}=r / (2v_{\eta_0})$, the probability in \eqref{eq:projection} is bounded by
$$
	2 \sum_{s=1}^{s_n/2 - s_0} N_s^{-1} e^{uv_{\eta_0}(s-s_0)} \left( e^{c(s-s_0)/2} + C \right).
$$
This tends to zero because
$$
	N_s \geq \frac{(p-s)^{s-s_0}}{\Gamma(s-s_0+1)} \geq \frac{(p/2)^{s-s_0}}{\Gamma(s-s_0+1)} \geq e^{(s-s_0)\log(p/2s_n)}
$$
and $s_n/p \leq s_n \lambda / \sqrt{n} = o(1)$.
\qed

\section{Discussion}
\label{sec-discussion}

Dimension  conditions such as $s_0 \log p \ll n^{1/6}$ are required for two reasons.
The first one is for handling the remainder term in the LAN expansion.
In this paper, we applied a bracketing argument to handle uniform convergence of empirical processes, but more elaborate chaining techniques as in \cite{spokoiny2012parametric} might be helpful to improve the required dimension.
In some parametric models, LAN holds under the dimension condition that $s_0 \ll n^{1/3}$ \cite{panov2015finite}, and  $n^{1/3}$ cannot be improved in general.
The critical dimension depends on the model, and $s_0 \ll n^{1/6}$ is required even in some parametric models \cite{ghosal2000asymptotic}.
The second reason is for handling the semi-parametric bias as explained in Section \ref{ssec:bvm}.
This part can be improved if we can estimate the score function $\score_\eta$ with a faster rate.
The rate for $\score_\eta$ is obtained using a Hellinger rate and structures of normal mixtures (Lemma \ref{lem:score-unif-rate}), which perhaps leaves some space for improvement. 
For the prior $\Pi_\cH$, we assumed that the base measure of the Dirichlet process is compactly supported.
This is mainly due to technical convenience, and with more delicate consideration using sieves, we believe that most results in this paper can be extended to more general priors.
Finally, it should be noted that the key property  one utilizes  for proving selection consistency in Theorem \ref{thm:selection} is the sub-Gaussianity of $\score_{\eta_0}(Y_i - x_i^T \theta_0)$.
The  proof can be extended to more general settings such as generalized linear models and Gaussian models under misspecification, if the corresponding score functions are sub-Gaussian.
Similar conditions can be found in frequentist's selection criteria \cite{kim2016consistent}.

\section{Acknowledgment}

Part of the research of MC and LL  was funded by  NSF grant  IIS1546331 and a grant from the Army's research office while DD's contribution was funded by ONR grant N00014-14-1-0245.  LL thanks David Pollard for pointing  her to useful references on bracketing. 

\appendix

\section{Empirical process with bracketing}\label{sec:maximal}

This section introduces bracketing methods for independent but not identically distributed random variables.
Suppose $Z_1, Z_2, \ldots$ is a sequence of independent $\cX$-valued random variables and let $\cF$ be a class of real-valued functions on $\cX$.
Let $N_{[]}^n(\delta, \cF)$ be the minimal number $N$ of sets in a partition $\{\cF_1, \ldots, \cF_N\}$ of $\cF$ such that
$$
	\frac{1}{n} \sum_{i=1}^n \E \sup_{f,g\in\cF_j} |f(Z_i) - g(Z_i)|^2 \leq \delta^2
$$
for every $j \leq N$.
For each $j$, fix $f_j \in \cF_j$, and let $A_\delta f (x) = f_j(x)$ and
$$
	B_\delta f(x) = \sup_{g,h\in\cF_j} |g(x) - h(x)|
$$
for $f \in \cF_j$.
Also let $B(x) = \max_{f\in\cF} B_\delta f(x)$.
Let $\bbG_n f = n^{-1/2} \sum_{i=1}^n (f(Z_i) - \E f(Z_i))$.
We always assume that
$$
	\lim_{\epsilon \rightarrow 0} N_{[]}^n(\epsilon, \cF) = \infty
	\quad \textrm{and} \quad
	\int_0^1 \sqrt{\log (N_{[]}^n(\epsilon, \cF))} < \infty.
$$

\begin{lemma}\label{lem:maximal_bracket}
For some universal constant $C > 0$,
\bean
	\E \sup_{f\in\cF} | \bbG_n(f-A_\delta(f))| 
	\leq C \int_0^\delta \sqrt{\log (2 N_{[]}^n(\epsilon, \cF))} d\epsilon
	+ \frac{1}{\sqrt{n}} \sum_{i=1}^n \E B(Z_i) 1_{\{B(Z_i) > \sqrt{n} \alpha \}},
\eean
where $\alpha = \delta / \sqrt{\log (2 N_{[]}^n(\delta, \cF))}$.
\end{lemma}
\begin{proof}
See \cite{pollard2001bracketing}.
A simpler proof for \iid\;cases can be found in \cite{van1998asymptotic}, Lemma 19.34.
\end{proof}

\medskip
\begin{corollary}\label{cor:maximal_bracket}
Assume that there is an envelop function $F$ of $\cF$ such that $\|F\|_n \leq 1$ for every $n$, where $\|F\|_n^2 = n^{-1} \sum_{i=1}^n \E F^2(Z_i)$.
Then, for some universal constant $C > 0$,
\bean
	\E \sup_{f\in\cF} |\bbG_n f| \leq C \int_0^{\|F\|_n} \sqrt{\log N^n_{[]}(\epsilon, \cF)} d\epsilon.
\eean
\end{corollary}
\begin{proof}
Let $\delta = 2 \|F\|_n$, then $N^n_{[]}(\delta, \cF)=1$ because $-F \leq f \leq F$ for every $f\in\cF$.
Let $B(x) = \sup_{f,g\in\cF} |f(x)-g(x)|$ and $\alpha = \delta/ \sqrt{\log 2}$.
Then, by the Cauchy-Scwartz and Markov's inequalities,
\bean
	\frac{1}{\sqrt{n}} \sum_{i=1}^n \E B(Z_i) 1_{\{B(Z_i) > \sqrt{n} \alpha \}}
	\leq \frac{1}{\sqrt{n}} \sum_{i=1}^n \sqrt{\E B^2(Z_i) \E 1_{\{B(Z_i) > \sqrt{n} \alpha \}}} 
	\\
	\leq \frac{1}{n\alpha} \sum_{i=1}^n \E B^2(Z_i) 
	\lesssim \frac{1}{n\alpha} \sum_{i=1}^n \E F^2(Z_i) \lesssim \delta^2 /\alpha \lesssim \delta.
\eean
Note that $\E \sup_{f\in\cF} |\bbG_n f| \leq \E \sup_{f\in\cF} |\bbG_n (f-g)|+ \E|\bbG_n g|$ for any $g\in \cF$.
Since 
\bean
	\E |\bbG_n g| \leq \sqrt{\E |\bbG_n g|^2}
	\leq \sqrt{\frac{1}{n} \sum_{i=1}^n \E g^2(Z_i)} \leq \|F\|_n = \delta/2,
\eean
we have
\bean
	\E \sup_{f\in\cF} |\bbG_n f| \lesssim	\int_0^\delta \sqrt{\log N^n_{[]}(\epsilon, \cF)} d\epsilon + \delta
	\lesssim \int_0^{2\|F\|_n} \sqrt{\log N^n_{[]}(\epsilon, \cF)} d\epsilon + \|F\|_n
	\\
	\lesssim \int_0^{\|F\|_n} \sqrt{\log N^n_{[]}(\epsilon, \cF)} d\epsilon + \|F\|_n
\eean
by Lemma \ref{lem:maximal_bracket}, where the last inequality holds by the monotonicity of $\epsilon \mapsto N_{[]}^n(\epsilon, \cF)$.
\end{proof}

\section{Concentration inequalities}

We state the Bernstein and Hanson-Wright inequalities for reader's convenience.

\begin{lemma}[Bernstein inequality] \label{lem:bernstein-inequality}
Let $Z_1, \ldots, Z_n$ be independent random variables with zero mean such that
$$
	K^2 \E\left(e^{|Z_i|/K} - 1 - \frac{|Y_i|}{K}\right) \leq \half v_i
$$
for some constants $K>0$ and $v_i$.
Then,
$$
	P_0\left(\left|Z_1 + \cdots + Z_n \right| > x \right) \leq 2e^{-\frac{x^2}{2(v + Kx)}},
$$
for $v \geq v_1 + \cdots v_n$.
\end{lemma}
\begin{proof}
See Lemma 2.2.11 in \cite{van1996weak}.
\end{proof}

\medskip
For a random variable $Z$ and the function $\psi_2(t) = e^{t^2} -1$, let
$$
	\|Z\|_{\psi_2} = \inf\left\{K > 0: E e^{Z^2 / K^2} \leq 2\right\}
$$
be the Orlicz norm.
If $\|Z\|_{\psi_2}$ is finite, then $Z$ is called a \emph{sub-Gaussian} random variable.
The Hanson-Wright inequality \cite{hanson1971bound, wright1973bound} provides a tail bound for a quadratic form of sub-Gaussian random variables.
For a matrix $A = (a_{ij})$, let $\|A\| = \sup_{x \neq 0} \|Ax\|_2 / \|x\|_2$ be the $\ell_2$-operator norm and $\|A\|_F = (\sum_{i,j} a_{ij}^2 )^{1/2}$ be the Frobenius norm.

\begin{lemma}[Hanson-Wright inequality] \label{lem:hw-inequality}
Let $Z = (Z_1, \ldots, Z_n)^T$ be a random vector whose components are independent and satisfy $\E Z_i = 0$ and $\|Z_i \|_{\psi_2} \leq K$.
Let $A$ be an $n\times n$ matrix.
Then, for some universal constant $C > 0$,
$$
	P_0\left( |Z^T A Z - \E Z^T A Z| > t \right) \leq 2 \exp\left\{ -C \min\left( \frac{t^2}{K^4 \|A\|_F^2}, \frac{t}{K^2 \|A\|} \right) \right\}
$$
for every $t\geq 0$.
\end{lemma}
\begin{proof}
See Theorem 1.1 in \cite{rudelson2013hanson}.
\end{proof}




\end{document}